\documentclass[11pt,a4paper]{article}

\usepackage[T1]{fontenc}
\usepackage[utf8]{inputenc}
\usepackage[english]{babel}
\usepackage{makecell}
\usepackage[a4paper,margin=1in]{geometry}
\usepackage{setspace}
\usepackage{parskip}

\usepackage{amsmath,amssymb,amsfonts,amsthm,mathtools}
\usepackage{bm}

\usepackage{graphicx}
\usepackage{float}
\usepackage{subcaption}
\usepackage{booktabs}
\usepackage{multirow}
\usepackage{array}
\usepackage{tabularx}

\usepackage{algorithm}
\usepackage{algpseudocode}

\usepackage[numbers,sort&compress]{natbib}
\usepackage[colorlinks=true,linkcolor=blue,citecolor=blue,urlcolor=blue]{hyperref}
\usepackage[nameinlink,capitalise,noabbrev]{cleveref}

\usepackage{xcolor}
\usepackage{orcidlink}
\usepackage{authblk}

\newtheorem{theorem}{Theorem}[section]
\newtheorem{lemma}[theorem]{Lemma}

\theoremstyle{definition}
\newtheorem{definition}[theorem]{Definition}

\title{\textbf{Robust Chance-Constrained Complex Zero-Sum Games}}

\author[1]{Raneem Madani\,\orcidlink{0009-0006-6757-0181}}
\author[1,2]{Abdel Lisser\,\orcidlink{0000-0003-1318-6679}}
\author[1]{Zeno Toffano\,\orcidlink{0000-0001-8594-3291}}

\affil[1]{Laboratoire des Signaux et Syst\`emes (L2S), CNRS, CentraleSup\'elec, Universit\'e Paris-Saclay, 91190 Gif-sur-Yvette, France}
\affil[2]{F\'ed\'eration des Math\'ematiques de CentraleSup\'elec, Universit\'e Paris-Saclay, 91190 Gif-sur-Yvette, France}

\date{April 2026}

\begin{document}

\maketitle

\begin{center}
Raneem Madani: \texttt{raneem.madani@centralesupelec.fr} \\
Abdel Lisser: \texttt{abdel.lisser@centralesupelec.fr} \\
Zeno Toffano: \texttt{zeno.toffano@centralesupelec.fr}
\end{center}

\begin{abstract}
This paper develops a unified framework for zero-sum games in which both the pure strategies and the payoff matrices contain complex-valued entries. By leveraging a linear isomorphism between complex and real vector spaces, we extend key results from real-valued convex analysis to the complex domain, establishing the validity of the minimax theorem and the preservation of saddle-point structure. Building on this foundation, we formulate a complex zero-sum game model that enables mixed strategies to interact with the real and imaginary components of the payoff matrix, and we characterize its saddle-point equilibrium through associated primal and dual problems. To incorporate uncertainty, we introduce a complex chance-constrained zero-sum game model that handles individual probabilistic constraints defined by complex linear functionals. We first study this formulation under known exact distributions, focusing on complex elliptically symmetric random variables, which generalize the complex Gaussian family. The framework is then extended to moment-based ambiguity sets, including: (i) distributions with known first two moments, (ii) distributions with unknown second-order moments, and (iii) distributions with fully unknown moments. In all cases, the probabilistic constraints admit deterministic second-order cone representations, ensuring convex feasible strategy sets and enabling an explicit characterization of the complex game value. Numerical experiments, including a transmitter--jammer waveform interaction model, show how the proposed framework captures the behavior of complex mixed strategies. We also evaluate out-of-sample violation rates and confirm that the empirical behavior is consistent with the theoretical guarantees.
\end{abstract}

\section{Introduction}\label{sec:1}
Optimization problems with complex variables are increasingly important because they underpin a wide range of real-world applications, including signal processing \cite{10452289,10.1007/978-3-030-27192-3_1}, quantum information and communication \cite{PhysRevA.109.012609}, and chemistry \cite{zheng2024unleashed}. A substantial body of work has extended classical optimization theory to the complex domain. Early contributions include formulations of complex linear programming \cite{LEVINSON196644}, followed by advances in complex nonlinear programming \cite{ABRAMS1972619, FERRERO1992399} and duality theory \cite{Craven_Mond_1973, DATTA19841}. Later developments, such as Cauchy–Riemann (CR) calculus \cite{kreutz2009complex} and algorithms for unconstrained complex optimization \cite{doi:10.1137/110832124}, further strengthened this theoretical and algorithmic foundation. As practical optimization problems increasingly incorporated uncertainty, stochastic optimization emerged \cite{dantzig1955linear}. In \cite{charnes1959chance}, Charnes and Cooper introduced chance-constrained programming (CCP), which enforces constraints probabilistically to guarantee feasibility at a prescribed confidence level. This paradigm remains a standard tool for balancing feasibility and risk in optimization under uncertainty \cite{CHENG2012325}. More recently, in \cite{10.1007/978-3-032-13589-6_21}, we investigated complex chance-constrained programming (3CP), focusing on linear 3CP under complex normal distributions.

In stochastic CCP, two main settings are typically considered. The first assumes that the distribution is known exactly, for instance, when it follows a complex elliptically symmetric (CES) distribution. The latter form a broad class of complex-valued probabilistic models that generalize the complex normal distribution by preserving its elliptical geometry while allowing both heavier and lighter tails \cite{1502990,6263313}. They arise in various engineering applications \cite{1091566, conte1991modelling} and include important special cases such as the complex Gaussian distribution \cite{goodman1963statistical}, the complex Student-$t$ distribution \cite{krishnaiah1986complex}, and the generalized Gaussian distribution \cite{novey2009complex}. The second setting arises when the exact probability distribution of the uncertainty is unknown. In this case, {Distributionally Robust Optimization (DRO)} extends stochastic optimization by relying only on partial statistical information about the uncertainty \cite{calafiore2006distributionally}. Instead of assuming a fixed distribution, DRO seeks solutions that perform well under the worst-case distribution within an {ambiguity set}, a family of distributions characterized by available information such as moments, support, or distance metrics \cite{rahimian2022frameworks, nguyen2024distributionally, doi:10.1137/130915315}.

Zero-sum games are a cornerstone of game theory, modeling competitive interactions in which one player’s gain is exactly the other’s loss. Early foundations were laid by \cite{Neumann1928}, whose minimax theorem established the existence of equilibrium in real-valued strategy spaces. This was later generalized by \cite{1103040253}, who showed that a saddle-point equilibrium exists for convex–concave functions over compact convex strategy sets. In parallel, \cite{nash1950equilibrium} demonstrated that every finite strategic game admits a Nash equilibrium, formalizing the modern equilibrium concept. These contributions exposed a deep connection between game theory and optimization: saddle-point equilibria of zero-sum games correspond to optimal solutions of primal–dual linear programs \cite{dantzig1951proof}. Initially, the literature focused on unconstrained games, but \cite{charnes1953constrained} extended this framework to constrained zero-sum games, proving that equilibrium can still be computed via primal–dual formulations when linear inequalities restrict players’ strategies. Randomness in payoffs and constraints was introduced by \cite{blau1974random} through the notion of random-payoff games, which later evolved into chance-constrained game models \cite{CHENG2016213, 10269785}. A substantial body of work followed, including characterizations of Nash equilibria under random payoffs \cite{riccardi2023complementarity,SINGH2016640,singh2018characterization, zhang2022variational}, joint chance constraints in general-sum games \cite{PENG2018482}, and extensions to mixture distributions \cite{peng2021chance}. Further developments incorporated dependent and elliptically distributed uncertainties \cite{nguyen2022random}, distributionally robust formulations under moment-based and distance-based ambiguity sets \cite{singh2017distributionally, SINGH2021109092}, and even Wasserstein-ball ambiguity sets \cite{XIA2023315}. 

In contrast, zero-sum games defined directly in complex spaces remain far less explored. A naive real-valued reformulation obtained by splitting variables into real and imaginary parts is often unsatisfactory, since only the real part admits a clear probabilistic interpretation, while the imaginary part is merely an auxiliary degree of freedom. Early work~\cite{mond1982game, murray1983solution, mond1983minimax} instead developed minimax results natively in the complex domain by introducing {complex probability vectors}: their real parts form standard nonnegative weights summing to one, and their imaginary parts are restricted to pointed polyhedral cones. In this framework, so-called {pure strategies} are not Dirac vectors but extreme feasible complex strategy vectors, and mixed strategies are obtained by taking real convex combinations of these extremes, effectively a two-level representation (real probabilities over complex probability vectors). This reduction enables solving an equivalent real matrix game via linear programming and mapping the solution back to the complex space. However, the approach ties the feasible mixed-strategy set to the particular cone generated from the chosen pure vectors, and it doesn't yield an explicit primal-dual saddle-point characterization.

This paper fills these gaps by developing a complete theoretical and algorithmic framework for complex zero-sum games under uncertainty. We begin by extending fundamental tools from convex analysis to the complex domain through a real–complex linear isomorphism, ensuring that minimax principles and saddle-point preservation carry over directly from real-valued settings.

We introduce norm-bounded complex strategy sets in which the real part forms a valid probability distribution and the imaginary part is intrinsically coupled through a norm constraint. This construction yields compact, convex feasible sets and eliminates the unboundedness issues present in prior models. Within this setting, we incorporate linear constraints and prove the existence of saddle-point equilibria, showing that primal–dual characterizations remain valid in the complex domain. To handle uncertainty, we integrate this framework with 3CP. We study several ambiguity sets of uncertainty: (i) random payoffs following Complex Elliptically Symmetric (CES) distributions; (ii) distributionally robust settings in which the first and second moments are known; (iii) ambiguity set with unknown second order moment, which accounts for the information about a mean vector and an upper bound on the second moments matrices  (iv) ambiguity set with unknown moments, where the mean vector lies in an ellipsoid and the second moment lies in a positive semi-definite cone. In each case, we derive convex deterministic reformulations, typically SOCP representations, that preserve feasibility, convexity, and tractability. This allows us to establish both the existence and the explicit primal–dual SOCP characterization of saddle-point equilibria, and to certify equilibrium strategies via their associated dual variables. To connect the theoretical results with practice, we include a numerical study that highlights how complex-valued mixed strategies behave under the proposed feasible-set construction and chance-constrained modeling. In particular, we use a transmitter–jammer waveform interaction model and report sensitivity analysis to assess empirical calibration against the target confidence levels.

The paper is structured as follows. Section~\ref{sec:2} introduces the notations and presents the main preliminaries required throughout the paper. Section~\ref{sec:3} develops the zero-sum game with a complex-valued payoff matrix under linear constraints and establishes the existence and characterization of the saddle-point equilibrium via the primal–dual formulation. Section~\ref{sec:4} extends the framework to zero-sum games with random constraint parameters, leading to complex chance-constrained formulations. We analyze multiple scenarios and prove both the existence and the characterization of the resulting saddle points. Section~\ref{sec:5} reports numerical experiments illustrating the proposed results. Finally, Section~\ref{sec:6} concludes the paper.
\section{Notations and Definitions}\label{sec:2}
\begin{table}
\caption{Summary of notation used throughout the paper}
\renewcommand{\arraystretch}{1.25}
\begin{tabular}{|l|l|}
\hline
\textbf{Symbol} & \textbf{Meaning} \\ \hline
$a = a_R + i a_I \in \mathbb{C}^n$   
& Complex vector \\ \hline
$\operatorname{Re}(a) = a_R,\operatorname{Im}(a) = a_I$       
& Real and imaginary parts of $a$ \\ \hline
$\varphi_1(a)=[a_R ~ a_I]^T$ 
& The isomorphism map of the complex vector from $\mathbb{C}^n\mapsto\mathbb{R}^{2n}$\\ \hline
$\bar{a}, a^T, a^H$                           
& Complex conjugate, Transpose and Hermitian (conjugate) transpose of $a$ \\ \hline
$A \in \mathbb{C}^{n \times m}$     
& Payoff matrix \\ \hline
$\varphi_2(A)=\begin{pmatrix} A_R &-A_I\\A_I&A_R\end{pmatrix}$ 
& The isomorphism map of complex matrix from $\mathbb{C}^{n\times m}\mapsto\mathbb{R}^{2n\times2m}$\\ \hline
$u \in U \subset \mathbb{C}^n,v\in V \subset \mathbb{C}^m$ 
& Mixed-strategy sets of player~1 and player~2 \\ \hline
$\alpha_1,\ \alpha_2$ 
& Norm-bound parameters of the mixed strategies of player~1 and player~2 \\ \hline
$B \in \mathbb{C}^{l \times n},\ b \in \mathbb{C}^l$ 
& Constraint parameters of player~1 \\ \hline
$D \in \mathbb{C}^{q \times m},\ d \in \mathbb{C}^q$ 
& Constraint parameters of player~2 \\ \hline
$S_1,\ S_2$ 
& Feasible sets of player~1 and player~2 \\ \hline
$\mu_B,\ \Gamma_B,\ J_B$ 
& Mean, covariance, and pseudo-covariance of random variable $B$ \\ \hline
$p_1,\ p_2$ 
& Probability vectors of the chance constraints for both players \\ \hline
$\Phi^{-1}(\cdot)$ 
& Quantile function (inverse cumulative distribution) \\ \hline
$\lambda_1,\beta_1, \rho_1,r_1$ 
& Lagrange multipliers of the primal problem \\ \hline
$\lambda_2,\beta_2, \rho_2,r_2$ 
& Lagrange multipliers of the dual problem \\ \hline
\end{tabular}
\label{tab:notation}
\end{table}
In this section, we introduce the notations, definitions, and mathematical tools used throughout the paper. We begin by presenting Table~\eqref{tab:notation}, which summarizes all notations employed in the sequel. We then review fundamental concepts related to complex variables and complex random variables. Next, we develop a linear isomorphism that maps complex spaces to their real representations, enabling us to extend several classical results to the complex domain. Building on this isomorphism, we present key convex-analytic tools for complex spaces and use them to establish a minimax theorem, prove saddle-point preservation, and derive additional constructions such as the dual norm in the complex setting. The set of complex vectors is denoted by $\mathbb{C}^n$. A complex number can be written as $a=(a_R,a_I)=a_R+ia_I$, where $a_R=\operatorname{Re}(a)$, $a_I=\operatorname{Im}(z)$ are the real and imaginary part of $a$ respectively \cite{ahlfors1979complex}. The complex conjugate of $a$ is defined as $\bar{a}=(a_R,-a_I)=a_R-ia_I$. The transpose and the complex conjugate transpose of $a$ are denoted by $a^T$ and $a^H$, respectively. The Euclidean norm of $a$ is equal to \[\|a\|=\sqrt{aa^H}=\sqrt{a_R^2+a_I^2}.\]
In probability theory and statistics, a complex random vector is typically a tuple of complex-valued random variables, and generally is a random variable taking values in a vector space over the field of complex numbers. Complex random variables can always be considered as pairs of real random vectors: their real and imaginary parts. A complex random vector $a\in\mathbb{C}^n$ on the probability space $(\Omega, \mathcal{F}, P)$ is a mapping $a: \Omega \rightarrow \mathbb{C}^{n}$ such that the vector $[\operatorname{Re}(a_{1}), \operatorname{Im}(a_{1})]^{T}$ is a real random vector on $(\Omega, \mathcal{F}, P)$. The random vector has a mean, a covariance matrix, and a pseudo-covariance matrix defined as follows \cite {1179767}:
\begin{definition}\label{Def: 1} The mean, covariance matrix, and pseudo-covariance matrix of the random vector $a$ are given by:
\begin{align}
    & \mu_a = \mathbb{E}[a] = \mathbb{E}[a_R]+ i\mathbb{E}[a_I]\\
    &\Gamma_a = Cov(a,a)=\mathbb{E}\left[(a-\mathbb{E}[a])(a-\mathbb{E}[a])^H\right]= \Gamma_{a_R} + \Gamma_{a_I} +i(\Gamma_{a_Ia_R}-\Gamma_{a_Ra_I})\\
    & J_a = Cov(a,\bar{a})= \mathbb{E}\left[(a-\mathbb{E}[a])(a-\mathbb{E}[a])^T\right] = \Gamma_{a_R}- \Gamma_{a_I}+ i(\Gamma_{a_Ia_R} + \Gamma_{a_Ra_I}) 
\end{align}
\end{definition}
The covariance matrix $\Gamma_a$ is Hermitian positive semidefinite $\Gamma_a^H=\Gamma_a$, and the pseudo-covariance matrix $J_a$ is symmetric $J_a^T=J_a$. Also $ \forall\alpha\in\mathbb{C}, Cov(a,b)=\overline{Cov(b,a)}$, $Cov(\alpha a,b)=\alpha Cov(a,b)$, and $Cov(a,\alpha b) = \bar{\alpha}Cov(a,b)$.

\begin{definition}[Proper complex random vectors \cite{1179767}]
    A complex random vector $a\in \mathbb{C}^n$ is called {proper} if the following three conditions are satisfied:
\begin{enumerate}
    \item Zero mean: $\mathbb{E}[a] = 0$,
    \item Finite variance for all components,
    \item Vanishing pseudo-covariance: $\mathbb{E}[aa^{T}] = 0$.
\end{enumerate}
\end{definition}
The following two linear mappings provide a bridge between complex and real spaces:
\begin{enumerate}
    \item For all $a \in \mathbb{C}^n$, define
    \begin{align}
        \varphi_1(a) &=
        \begin{pmatrix}
        {Re}(a) \\
        {Im}(a)
        \end{pmatrix}
        \in \mathbb{R}^{2n}.
        \label{eq:phi1}
    \end{align}
    \item For all $A \in \mathbb{C}^{n \times m}$, define
    \begin{align}
        \varphi_2(A) &=
        \begin{pmatrix}
        {Re}(A) & -{Im}(A) \\
        {Im}(A) & {Re}(A)
        \end{pmatrix}
        \in \mathbb{R}^{2n \times 2m}.
        \label{eq:phi2}
    \end{align}
\end{enumerate}
These mappings are one-to-one and preserve the algebraic structure of addition, scalar multiplication, matrix multiplication, transposition, and inversion. Furthermore, in \cite{ZHANG201559}, the authors proved the following useful properties and lemmas. For any
\( A \in \mathbb{C}^{n \times m} \), \( B \in \mathbb{C}^{n \times k} \),
\( D \in \mathbb{C}^{n \times n} \) where \( D \) is invertible,
\( a, b \in \mathbb{C}^n \), \( \sigma \in \mathbb{R} \), and \( I_n \in \mathbb{R}^{n \times n} \) is the identity matrix, we have:
\begin{enumerate}
    \item $\varphi_1(a \pm b) = \varphi_1(a) \pm \varphi_1(b), \quad \varphi_1(\sigma a) = \sigma \varphi_1(a),$
    \item $\varphi_2(A \pm B) = \varphi_2(A) \pm \varphi_2(B), \quad \varphi_2(\sigma A) = \sigma \varphi_2(A), \quad \varphi_2(I_n) = I_{2n},$
    \item $\varphi_2(AB) = \varphi_2(A)\varphi_2(B), \quad \varphi_1(Aa) = \varphi_2(A)\varphi_1(a),$
    \item $\varphi_2(A^H) = \left( \varphi_2(A) \right)^T, \quad \varphi_2(D^{-1}) = \left( \varphi_2(D) \right)^{-1},$
    \item $\|a\|^2 = \|\varphi_1(a)\|^2, \quad {Re}(a^H b) = \varphi_1(a)^T \varphi_1(b).$
\end{enumerate}
\begin{lemma}[\cite{ZHANG201559}]\label{lemm:convex-image}
If $U \subset \mathbb{C}^n$ is convex, then its image $\tilde{U} := \{\varphi_1(u) : u \in U\} \subset \mathbb{R}^{2n}$ is also convex.
\end{lemma}
\begin{lemma}[\cite{ZHANG201559}]\label{lemm:convexity-preservation}
Let $g:\mathbb{C}^n \to \mathbb{R}$ and define $f(w) := g(\varphi_1^{-1}(w)).$ If $g$ is convex (respectively concave) on a convex set $U \subset \mathbb{C}^n$, 
then $f$ is convex (respectively concave) on $\tilde{U}$.
\end{lemma}
\begin{lemma}[\cite{arkhangel1990basic}]\label{lemm:compact-image}
If $U \subset \mathbb{C}^n$ is compact, then its image $\tilde U := \{\varphi_1(u):u\in U\}$ is also compact.
\end{lemma}
We now state the following lemma, in which we establish the saddle-point preservation and the minimax theorem in the complex domain.
\begin{lemma}[Saddle-Point Preservation]
\label{lem:saddle-preserve}
Let $S_1\subset \mathbb{C}^n, S_2\subset\mathbb{C}^m,$ and let $g:S_1\times S_2\to\mathbb{R}$, define $f(\tilde u,\tilde v) 
    := g(\varphi_1^{-1}(\tilde u), \varphi_1^{-1}(\tilde v)).$ Then $(u^\star,v^\star)\in S_1\times S_2$ is a saddle point of $g$ iff $(\varphi_1(u^\star),\varphi_1(v^\star))$ is a saddle point of $f$.
\end{lemma}
\begin{proof}
Suppose $(u^\star,v^\star)$ is a saddle point of $g$, i.e.,
\begin{align}
    g(u^\star,v) \leq g(u^\star,v^\star) \leq g(u,v^\star),
    \quad \forall u \in S_1, v \in S_2.
\end{align}
Let $\tilde u=\varphi_1(u)$, $\tilde v=\varphi_1(v)$. Then
\begin{align}
    f(\varphi_1(u^\star),\tilde v) = g(u^\star,v) \leq g(u^\star,v^\star) =f(\varphi_1(u^\star),\varphi_1(v^\star))\leq g(u,v^\star) = f(\tilde u,\varphi_1(v^\star)).
\end{align}
Hence $(\varphi_1(u^\star),\varphi_1(v^\star))$ is a saddle point of $f$.  
The converse follows by reversing the argument.
\end{proof}
Sion’s minimax theorem \cite{1103040253} extends von Neumann’s classical minimax result \cite{Neumann1928} by relaxing the requirement that strategies lie in simplexes and that the payoff function be bilinear. Instead, the theorem applies to general convex and concave functionals over convex and compact sets. We now present a complex-valued analogue of the minimax theorem in \cite{1103040253}:
\begin{theorem}[Minimax theorem in Complex Space]
Let $U \subseteq \mathbb{C}^n$ and $V \subseteq \mathbb{C}^m$ be compact convex sets. If $g: U \times V \to \mathbb{R}$ is a continuous and concave-convex function, then there exists a saddle point $(u^\star,v^\star)$ and the minimax equality holds:
\begin{align}
    \max_{u \in U} \min_{v \in V} g(u,v)
    &= \min_{v \in V} \max_{u \in U} g(u,v).
    \label{eq:minimax}
\end{align}
Moreover, $(u^\star,v^\star)$ is a saddle point if
\begin{align}
    g(u, v^\star) \leq g(u^\star, v^\star) \leq g(u^\star, v), 
    \quad \forall u \in U, v \in V.  \label{eq:saddle-condition}
\end{align}\label{Th: minimax}
\end{theorem}
\begin{proof}
Define $\tilde U := \{\varphi_1(u):u\in U\},\quad
\tilde V := \{\varphi_1(v):v\in V\}.$ By lemmas~\eqref{lemm:convex-image} and \eqref{lemm:compact-image}, 
$\tilde U$ and $\tilde V$ are convex and compact. Define $f(\tilde u,\tilde v) := g(\varphi_1^{-1}(\tilde u),\varphi_1^{-1}(\tilde v))$.  
By lemma~\eqref{lemm:convexity-preservation}, if $g(\cdot,v)$ is concave then $f(\cdot,\tilde v)$ is concave, and if $g(u,\cdot)$ is convex then $f(\tilde u,\cdot)$ is convex. Thus $f$ is a real-valued concave-convex function on compact convex sets.  By Von Neumann’s minimax theorem in real spaces,
\begin{align}
    \max_{\tilde u \in \tilde U} \min_{\tilde v \in \tilde V} f(\tilde u,\tilde v)
    = \min_{\tilde v \in \tilde V} \max_{\tilde u \in \tilde U} f(\tilde u,\tilde v).
\end{align}
Mapping back to the complex domain gives \eqref{eq:minimax}. Finally, the existence of a saddle point $(\tilde u^\star,\tilde v^\star)$ for $f$ implies, by Lemma~\eqref{lem:saddle-preserve}, that $(u^\star,v^\star)=(\varphi_1^{-1}(\tilde u^\star),\varphi_1^{-1}(\tilde v^\star))$ is a saddle point of $g$, satisfying \eqref{eq:saddle-condition}.
\end{proof}
We now present the dual norm in the complex domain, which is essential for characterizing the saddle-point equilibrium.
\begin{theorem}[Dual of Norm of Complex Vector]
The dual of the norm $\|z\|$, $z\in \mathbb{C}^n$, can be expressed as
\begin{align}
    \|z\| \;=\; \sup_{u\in\mathbb{C}^n} 
    \Big\{ \operatorname{Re}(z^H u) \;\big|\; \|u\|\leq 1 \Big\} =\; \max_{\|u\|\leq 1} \operatorname{Re}(z^H u).
\end{align}\label{lem: dual norm}
\end{theorem}
\begin{proof}
Consider the real-linear isomorphism $\varphi_1:\mathbb{C}^n \to \mathbb{R}^{2n}$ defined by $\varphi_1(z) = \tilde z.$ In the real space, the dual norm is given by $\|\tilde z\| = \sup_{\tilde u \in \mathbb{R}^{2n}} 
    \Big\{ \tilde z^T \tilde u \;\big|\; \|\tilde u\|\leq 1 \Big\}
    = \max_{\|\tilde u\|\leq 1} \tilde z^T \tilde u$ \cite{rockafellar1970convex}. Using the identity $\tilde z^T \tilde u = \operatorname{Re}(z^H u)$ and noting that 
$\|\tilde u\| = \|u\|$, we obtain $\|z\| \;=\; \sup_{\|u\|\leq 1}\operatorname{Re}(z^H u).$ 
\end{proof}
\section{Zero-Sum Game with Complex Payoff Matrix}\label{sec:3}
In this section, we introduce the zero-sum game model with a complex-valued payoff matrix. We begin by defining the utility function and the mixed-strategy sets for both players and then incorporate complex linear constraints into each player's strategy space. After establishing this framework, we prove the existence of a saddle point and characterize it through the primal and dual formulations associated with each player.

\begin{definition}
A complex zero-sum game $G$ in strategic form is defined by the triple 
\[
G=(I,J,g),
\]
where $I\subset\mathbb{C}^n$ (resp.\ $J\subset\mathbb{C}^m$) is the nonempty strategy set of player~1 (resp.\ player~2), and 
$g:I\times J\to\mathbb{C}$ is the complex-valued game function. 
\end{definition}
Player~1 chooses $i\in I$ and player~2 chooses $j\in J$, independently (e.g., simultaneously).  
The payoff of player~1 is $g(i,j)$ and that of player~2 is $-g(i,j)$.  
Thus, the evaluations induced by the joint choice $(i,j)$ are opposite for the two players: player~1 is the maximizer and player~2 is the minimizer.

When $I$ and $J$ are finite, the game $G=(I,J,A)$ is represented by an $n\times m$ complex matrix $A$, where player~1 chooses a row $i\in I$, player~2 chooses a column $j\in J$, and the entry $A_{ij}$ corresponds to the game function $g(i,j)$.

In many situations, it is natural for players to randomize their behavior. A mixed strategy can also be interpreted as specifying the opponent's degree of belief that a given action will be selected. Mathematically, mixed strategies allow the analysis to take place in convex and compact sets and enable interactions between the real and imaginary parts of the game function and the probability vectors. To formalize randomized behavior, both players adopt {complex-valued mixed strategies} whose real parts form valid probability distributions. The feasible sets are
\begin{align}
U &:= \left\{ u\in\mathbb{C}^n : \|u\|\le\alpha_1,\ \operatorname{Re}(u)\ge 0,\ \sum_{i=1}^n u_i=1 \right\}, \label{4}\\
V &:= \left\{ v\in\mathbb{C}^m : \|v\|\le\alpha_2,\ \operatorname{Re}(v)\ge 0,\ \sum_{i=1}^m v_i=1 \right\}, \label{5}
\end{align}
for player~1 and player~2, respectively, where $\alpha_1\ge 1/\sqrt{n}$ and $\alpha_2\ge 1/\sqrt{m}$ are norm bounds controlling the allowable magnitude of complex strategies. The constraints $\operatorname{Re}(u)\ge0$ (resp.\ $\operatorname{Re}(v)\ge0$) and $\sum_i u_i=1$ (resp.\ $\sum_i v_i=1$) ensure that the real parts represent valid probability distributions over the $n$ (resp.\ $m$) pure actions.

The imaginary parts $\operatorname{Im}(u)$ and $\operatorname{Im}(v)$ do not carry probability mass, since their components sum to zero. Instead, they influence the payoff through interference terms in $\operatorname{Re}(u^HAv)$. The Euclidean norm constraints
\begin{align}
\|u\|^2 &= \|\operatorname{Re}(u)\|^2 + \|\operatorname{Im}(u)\|^2 \le \alpha_1^2,\\
\|v\|^2 &= \|\operatorname{Re}(v)\|^2 + \|\operatorname{Im}(v)\|^2 \le \alpha_2^2,
\end{align}
create a tradeoff between the magnitude of the imaginary and real components. The role of the bounds $\alpha_1,\alpha_2$ is:
\begin{itemize}
\item If $\alpha_1<1/\sqrt{n}$ (resp.\ $\alpha_2<1/\sqrt{m}$), the feasible sets are empty, since even the uniform distribution has norm $1/\sqrt{n}$ (resp.\ $1/\sqrt{m}$).
\item If $1/\sqrt{n}\le\alpha_1<1$ (resp.\ $1/\sqrt{m}\le\alpha_2<1$), real and imaginary parts are tightly coupled; using more phase freedom reduces the feasible region.
\item If $\alpha_1\ge1$ (resp.\ $\alpha_2\ge1$), then the norm bound does not restrict the real probabilities, since any probability vector satisfies $\|\operatorname{Re}(u)\|\le1$. The constraint reduces to $\|\operatorname{Im}(u)\|\le\alpha_1$. In this case, we can make a linear bound for the imaginary part $|\operatorname{Im}(u_i)|\le\alpha_1$. In this regime, the norm constraint becomes linear and the optimization problem remains a linear program.
\end{itemize}
The mixed extension of a finite game $G=(I,J,g)$ is the game
\[
\mathcal{G}=(U,V,g),
\]
where $g$ is extended multi-linearly:
\[
g(u,v)=\mathbb{E}_{u\otimes v}[g]=u^Hg(i,j)v=\sum_{i\in I}\sum_{j\in J} \bar{u}_i v_j g(i,j).
\]
A pure action $i\in I$ is identified with the Dirac vector $e_i\in U$, and similarly for player~2. The support of a mixed strategy $u$ is $\operatorname{supp}(u)=\{\,i:\operatorname{Re}(u_i)>0\,\}$. When $G$ is represented by a matrix $A$, a mixed strategy $u\in U$ corresponds to a row vector and $v\in V$ to a column vector, with complex payoff $g(u,v)=u^HAv$.
\begin{theorem}\label{thm:minmax}
Let $A$ be an $n\times m$ complex matrix. Then there exists $(u^{*},v^{*},\varrho)\in U\times V\times\mathbb{R}$ such that
\[
\operatorname{Re}(u^{*H}Av)\ge\varrho,\ \forall v\in V,
\qquad\text{and}\qquad
\operatorname{Re}(u^HAv^{*})\le\varrho,\ \forall u\in U.
\]
The real number $\varrho$ is uniquely determined and corresponds to the value of the game:
\[\varrho=\max_{u\in U}\min_{v\in V}\operatorname{Re}(u^HAv)=\min_{v\in V}\max_{u\in U}\operatorname{Re}(u^HAv)\]
\end{theorem}
\begin{proof}
Since the payoff function is bilinear and the mixed strategy spaces are convex and compact, the conditions of the Minimax Theorem~\eqref{Th: minimax} hold, and the result follows.
\end{proof}
For player~1, let $r_1\in\mathbb{R}_+^n$, $\rho_1\in\mathbb{C}$, and $\beta_1\in\mathbb{C}^n$ denote dual variables associated with the norm and linear constraints, and let $\mathbf{1}_n$ denote the $n$-vector of ones.  
Likewise, for player~2, let $r_2\in\mathbb{R}_+^m$, $\rho_2\in\mathbb{C}$, and $\beta_2\in\mathbb{C}^m$ be the corresponding dual variables. The two dual linear programs are
\begin{equation}
\begin{aligned}
\min_{v,\beta_1,\rho_1,r_1}
&\quad \alpha_1\|\beta_1\|+\operatorname{Re}(\rho_1)\\
\text{s.t.}\quad
& Av-\beta_1+r_1-\bar{\rho}_1\mathbf{1}_n=0,\\
& v\in V,\qquad r_1\ge0
\end{aligned}\tag{P1}\label{pb: P1}
\end{equation}
\begin{equation}
\begin{aligned}
\max_{u,\beta_2,\rho_2,r_2}
&\quad -\alpha_2\|\beta_2\|-\operatorname{Re}(\rho_2)\\
\text{s.t.}\quad
& A^H u+\beta_2-r_2+\bar{\rho}_2\mathbf{1}_m=0,\\
& u\in U,\qquad r_2\ge0
\end{aligned}\tag{D1}\label{pb: D1}
\end{equation}
We now incorporate additional linear constraints for each player. The constrained game is
\[
G=\max_{u\in S_1}\min_{v\in S_2}\operatorname{Re}(u^HAv)
\]
with feasible sets
\begin{align}
S_1 &= \{\,u\in U: \operatorname{Re}(Bu)\le b\,\}, \label{S1}\\
S_2 &= \{\,v\in V: \operatorname{Re}(Dv)\ge d\,\},\label{S2}
\end{align}
where $B\in\mathbb{C}^{l\times n}$, $D\in\mathbb{C}^{q\times m}$, $b\in\mathbb{R}^l$, and $d\in\mathbb{R}^q$.  
For fixed $v\in S_2$, player~1 solves
\begin{equation}
\begin{aligned}
\max_{u\in\mathbb{C}^n}\ &\operatorname{Re}(u^HAv)\\
\text{s.t.}\ &\operatorname{Re}(Bu)\le b,\\
&\|u\|\le\alpha_1,\ \operatorname{Re}(u)\ge0,\ \sum_{i=1}^n u_i=1,
\end{aligned}\label{23}
\end{equation}
and for fixed $u\in S_1$, player~2 solves
\begin{equation}
\begin{aligned}
\min_{v\in\mathbb{C}^m}\ &\operatorname{Re}(u^HAv)\\
\text{s.t.}\ &\operatorname{Re}(Dv)\ge d,\\
&\|v\|\le\alpha_2,\ \operatorname{Re}(v)\ge0,\ \sum_{i=1}^m v_i=1.
\end{aligned}\label{24}
\end{equation}
A pair $(u^\star,v^\star)$ is a {saddle-point equilibrium} if $u^\star$ solves \eqref{23} for fixed $v^\star$ and $v^\star$ solves \eqref{24} for fixed $u^\star$.

\textbf{Assumption 1 (Slater).} The sets $S_1$ and $S_2$ are strictly feasible: there exist $u\in S_1$ and $v\in S_2$ such that all defining inequalities hold with strict slack.

If the players choose purely real strategies ($\operatorname{Im}(u)=\operatorname{Im}(v)=0$), then for $\alpha_1>1$ and $\alpha_2>1$ (or equivalently, we remove the imagenary part from the mixed strategies), the feasible sets reduce to the classical simplices $\Delta^n$ and $\Delta^m$.  
In particular, when $\alpha_1=\alpha_2=1$, the model recovers the standard real zero-sum game.

\begin{theorem}
Let $S_1$ and $S_2$ be the feasible strategy sets defined in 
\eqref{S1}--\eqref{S2}, $f(u,v)=\operatorname{Re}(u^HAv)$, and let Assumption 1 hold. Then the complex zero sum game $G$ admits a saddle point equilibria $(u^{\star},v^{\star}) \in S_1 \times S_2$ such that
\begin{align}
    f(u^{\star},v) \le f(u^{\star},v^{\star}) \le f(u,v^{\star}), 
    \qquad \forall u \in S_1,\; v \in S_2.
\end{align}
\end{theorem}
\begin{proof}
By Theorem~\ref{Th: minimax}, it suffices that $S_1$ and $S_2$ are nonempty, compact, and convex, and that $f(u,v)=\operatorname{Re}(u^HAv)$ is continuous and concave in $u$ for fixed $v$ and convex in $v$ for fixed $u$. The sets $U$ and $V$ are closed and bounded subsets of $\mathbb{C}^n$ and $\mathbb{C}^m$, hence compact and convex. The sets $S_1$ and $S_2$ are obtained by intersecting $U$ and $V$ with finitely many closed halfspaces defined by the linear constraints $\operatorname{Re}(Bu)\le b$ and $\operatorname{Re}(Dv)\ge d$, and are therefore nonempty, compact, and convex as well. Therefore, the minimax conditions are satisfied, and a saddle-point equilibrium $(u^\star,v^\star)\in S_1\times S_2$ exists.
\end{proof}
The Lagrangian function of problem \eqref{23} is given by
\begin{align}
    \mathcal{L}(u;v,\lambda_1,s,\rho_1,r_1)
    &= \operatorname{Re}(u^HAv)
    - \sum_{i=1}^l \lambda_{1i}\, \operatorname{Re}\left(B_i u - b_i\right)
    - s\left(\|u\|-\alpha_1\right)
    + r_{1}^T \operatorname{Re}(u)
    + \operatorname{Re}\left(\rho_1 (1-\mathbf{1}_n^T u)\right),\label{lagr1}
\end{align}
where $u\in\mathbb{C}^n$ and $v\in S_2$ are the strategies, 
$\lambda_1\in\mathbb{R}_+^l$, $s\in\mathbb{R}_+$, $r_1\in\mathbb{R}_+^n$, 
$\rho_1\in\mathbb{C}$ are the Lagrange multipliers, and 
$\mathbf{1}_n$ denotes the $n$-dimensional vector of ones. Using the dual representation of the norm of a complex vector (see Lemma~\eqref{lem: dual norm}), we can express the contribution of the norm constraint via an auxiliary variable $\beta_1\in\mathbb{C}^n$ as
\begin{align*}
-s\bigl(\|u\|-\alpha_1\bigr)
&\;=\;\min_{\beta_1\in\mathbb{C}^n} \bigl\{-\operatorname{Re}(\beta_1^H u) + \alpha_1\|\beta_1\|\bigr\}.
\end{align*}
Substituting this into \eqref{lagr1} yields the equivalent Lagrangian
\begin{align*}
    \mathcal{L}(u;v,\lambda_1,\beta_1,\rho_1,r_1)
    &= \operatorname{Re}(u^HAv)
    - \sum_{i=1}^l \lambda_{1i}\, \operatorname{Re}\left(B_i u - b_i\right)
    - \operatorname{Re}(\beta_1^H u)
    + \alpha_1\|\beta_1\|
    + r_{1}^T \operatorname{Re}(u)
    + \operatorname{Re}\left(\rho_1 (1-\mathbf{1}_n^T u)\right)\\[0.5em]
    &= \operatorname{Re} \left((Av -B^H\lambda_1-\beta_1 + r_1 - \bar{\rho}_1 \mathbf{1}_n)^H u\right)
        + \lambda_1^T b + \alpha_1\|\beta_1\| + \operatorname{Re}(\rho_1),
\end{align*}
where $\beta_1\in\mathbb{C}^n$ is the dual variable associated with the norm constraint. Maximizing the Lagrangian over $u$ is finite if and only if the coefficient of $u$ vanishes:
\begin{equation}
    Av -B^H\lambda_1-\beta_1 + r_1 - \bar{\rho}_1 \mathbf{1}_n= 0.
\end{equation}
The resulting dual problem associated with \eqref{23} is then
\begin{equation}
    \begin{aligned}
    \min_{v,\lambda_1,\beta_1,\rho_1,r_1}
    &\quad  \lambda_1^T b + \alpha_1\|\beta_1\|+ \operatorname{Re}(\rho_1) \\
    \text{s.t.}\quad
    & Av -B^H\lambda_1-\beta_1 + r_1 - \bar{\rho}_1 \mathbf{1}_n= 0,\\
    & \operatorname{Re}(Dv)\ge d,\\
    & \|v\| \le \alpha_2,\quad \mathbf{1}_m^T v = 1,\quad \operatorname{Re}(v)\ge 0,\\
    & r_1 \ge 0,\quad  \lambda_1\ge 0.
    \end{aligned}\tag{P2}\label{pb: P2}
\end{equation}
By symmetry (exchanging the players’ roles), we obtain the dual problem for 
player~2’s:
\begin{equation}
    \begin{aligned}
    \max_{u,\lambda_2,\beta_2,\rho_2,r_2}
    &\quad \lambda_{2}^T d - \alpha_2\|\beta_2\| - \operatorname{Re}(\rho_2)\\
    \text{s.t.}\quad
    & A^Hu - D^H\lambda_{2} + \beta_2 - r_2 + \bar{\rho}_2\mathbf{1}_m = 0,\\
    & \operatorname{Re}(Bu)\le b,\\
    & \|u\| \le \alpha_1,\quad \mathbf{1}_n^T u = 1,\quad \operatorname{Re}(u)\ge 0,\\
    & r_2 \ge 0,\quad   \lambda_2\ge 0,
    \end{aligned}\tag{D2}\label{pb: D2}
\end{equation}
where $\lambda_2\in\mathbb{R}_+^q$, $r_2\in\mathbb{R}_+^m$, 
$\beta_2\in\mathbb{C}^m$, and $\rho_2\in\mathbb{C}$ are the corresponding dual variables. 
\begin{theorem}
Consider a complex constrained zero-sum game where the matrices 
$B \in \mathbb{C}^{l \times n}$ and $D \in \mathbb{C}^{q \times m}$ define the linear constraints of player~1 and player~2, respectively. Then, a pair $(u^\star,v^\star)$ is a saddle-point equilibrium if and only if there exist dual variables $(\rho_1^{\star}, r_1^{\star}, \lambda_{1}^{\star}, \beta_1^{\star})$ and $(\rho_2^{\star}, r_2^{\star}, \lambda_{2}^{\star}, \beta_2^{\star})$ such that  $(v^\star,\rho_1^{\star},r_1^{\star},\lambda_{1}^{\star},\beta_1^{\star})$ and $(u^\star,\rho_2^{\star},r_2^{\star},\lambda_{2}^{\star},\beta_2^{\star})$ are optimal solutions of the primal-dual pair \eqref{pb: P2}-\eqref{pb: D2}, respectively.
\end{theorem}
\begin{proof}
Let $(u^\star,v^\star)$ be a saddle-point equilibrium. Since \eqref{pb: P1} and \eqref{pb: D1} are convex SOCPs and strictly feasible, 
strong duality holds. Thus, their optimal values coincide:
\begin{align}
\lambda_{1}^{T\star} b_i +\alpha_1\|\beta_1^{\star}\| + \operatorname{Re}(\rho_1^{\star}) 
=
\lambda_2^{\star T} d - \alpha_2\|\beta_2^{\star}\| - \operatorname{Re}(\rho_2^{\star}).
\end{align}
Conversely, from the primal constraint of \eqref{pb: P2}, multiplying by $u$ and using $Bu \le b$ yields
\begin{align}
\operatorname{Re}(u^H A v^\star) \le \lambda_{1}^{T\star} b_i +\alpha_1\|\beta_1^{\star}\| + \operatorname{Re}(\rho_1^{\star}) , \quad \forall u \in S_1.\label{eq:6}
\end{align}
Similarly, from the dual constraint of \eqref{pb: D2}, multiplying by $v$ and using $Dv \ge d$ gives
\begin{align}
\operatorname{Re}(u^{\star H} A v) \ge \lambda_2^{\star T} d - \alpha_2\|\beta_2^{\star}\| - \operatorname{Re}(\rho_2^{\star}), \quad \forall v \in S_2.\label{eq:7}
\end{align}
Combining \eqref{eq:6}--\eqref{eq:7} yields, $\forall u \in S_1,\forall v \in S_2$:
\begin{align}
\operatorname{Re}(u^H A v^\star)
&\le \operatorname{Re}(\rho_1^{\star}) + \sum_i \lambda_{1i}^{\star} b_i + \alpha_1\|\beta_1^{\star}\|= \lambda_2^{\star T} d - \alpha_2\|\beta_2^{\star}\| - \operatorname{Re}(\rho_2^{\star})
\le \operatorname{Re}(u^{\star H} A v), 
\end{align}
Taking $u=u^\star$ and $v=v^\star$ gives equality throughout, which is exactly the saddle–point condition.  
\end{proof}
\section{Complex Zero-Sum Games with Complex Chance Constraints}\label{sec:4}
Real-world systems rarely operate in perfectly known environments. Constraints on strategies are often affected by measurement noise, model mismatch, or random perturbations. To reflect this, we extend the deterministic game of the previous section by introducing {complex chance constraints (3CP)}: each player must choose a strategy that satisfies its constraints with high probability. To unify notation for the random constraints of both players, we introduce the generic tuple
\[
(M,m,z,p)\in\{(B,b,u,p_1),\,(-D,-d,v,p_2)\},
\]
so that every probabilistic constraint in the game can be written in the compact form
\begin{align}
\mathbb{P}\!\left[ \operatorname{Re}(Mz)\le m\right]\ge p. \label{3CP}
\end{align}
Each row of the random vector $M$ is modeled as
\[
M_i \sim \mathcal{I}(\mu_{M_i},\Gamma_{M_i},J_{M_i}),\qquad i=1,\dots,r,
\]
where $r=\ell$ for Player 1 and $r=q$ for Player 2, and $\mathcal{I}$ denotes the ambiguity set of probability distributions. Thus, the optimization problem faced by Player~1 (with fixed $v$) is
\begin{align}
\max_{u\in U}\;& \operatorname{Re}(u^HAv) \\
\text{s.t.}\;&\mathbb{P}\!\left[ \operatorname{Re}(Bu)\le b\right]\ge p_1, 
\label{eq:cc-player1}
\end{align}
and symmetrically, Player~2 (with fixed $u$) solves
\begin{align}
\min_{v\in V}\;& \operatorname{Re}(u^HAv) \\
\text{s.t.}\;&\mathbb{P}\!\left[ \operatorname{Re}(Dv)\ge d\right]\ge p_2
\label{eq:cc-player2}
\end{align}
We therefore define the chance-constrained feasible sets in unified form:
\[S_1(p_1)=\{u\in U:\mathbb{P}[ \operatorname{Re}(Bu)\le b]\ge p_1\},\]
\[S_2(p_2)=\{v\in V:\mathbb{P}[ \operatorname{Re}(Dv)\ge d]\ge p_2\}.\]
The resulting zero-sum game is
\begin{align}
G(p):=\max_{u\in S_1(p_1)}\;\min_{v\in S_2(p_2)} \operatorname{Re}(u^HAv).
\label{Gp}
\end{align}
Without loss of generality, we assume that we have one chance constraint for each player, i.e., for $(M,m,z,p)\in\{(B,b,u,p_1),(-D,-d,v,p_2)\}$ and we study the constraint \eqref{3CP}.
\begin{lemma}[\cite{10.1007/978-3-032-13589-6_21}]
Let $M=(M_1,\dots,M_n)\in\mathbb{C}^n$ be a complex random vector with mean $\mu_M$, covariance $\Gamma_M$, and pseudo-covariance $J_M$. Then $ \operatorname{Re}(Mz)$ has real mean and variance:
\begin{align}
&\mu_{ \operatorname{Re}(Mz)}= \operatorname{Re}(\mu_M z),\\
&\Gamma_{ \operatorname{Re}(Mz)}
=\tfrac12\big(z^H\Gamma_M z+ \operatorname{Re}(z^T J_M z)\big).
\end{align}
Furthermore, the variance is a second-order function of $z$.
\end{lemma}
\subsection{Complex Chance-Constrained Games with Complex Elliptically Symmetric Distribution}
We now study the case in which the random constraint vector $M$ follows a {complex elliptically symmetric (CES)} distribution. A central assumption behind CES models is circularity, meaning that the distribution of a complex random vector is invariant under multiplication by any unit-modulus complex number. 
\begin{definition}[Complex Elliptically Symmetric (CES) Distribution {\cite{1502990}}]
A random vector $M\in\mathbb{C}^n$ with mean $\mu_M$, covariance $\Gamma_M$, 
and pseudo-covariance $J_M$ is said to follow a {complex elliptically 
symmetric} (CES) distribution if its characteristic function is
\begin{equation}
\Xi_M(z)=\exp\!\big(i\, \operatorname{Re}(z^H\mu_M)\big)\,
\psi\!\left(z^H\Gamma_M z +  \operatorname{Re}(z^H J_M\bar{z})\right),
\label{eq:CES}
\end{equation}
for some characteristic generator $\psi$. 
We write $M\sim\mathrm{CES}(\mu_M,\Gamma_M,J_M)$.
\end{definition}
In this case, the ambiguity set for each player is simply the CES family:
\[
\mathcal{I}(\mu_M,\Gamma_M,J_M)
=\mathrm{CES}(\mu_M,\Gamma_M,J_M),
\]
\begin{table}
\centering
\small
\caption{Representative subclasses of complex elliptically symmetric (CES) distributions.}
\label{CES}
\begin{tabular}{|l|c|c|c|}
\hline
\textbf{Distribution} & \textbf{Generator $\psi(t)$} & \textbf{Normalizing $C_g$} & \textbf{Special cases} \\
\hline
{Complex student-$t$} 
& $  \left(1+\frac{2t}{\nu}\right)^{-\frac{2m+\nu}{2}}  $ 
& $\frac{2^{m}\Gamma\left(\frac{2m+\nu}{2}\right)}{\left(\pi\nu\right)^{m}\Gamma\left(\frac{\nu}{2}\right)}$ 
& $\nu{=}1$: Cauchy\\
${CT}_{\nu,m}({\mu},{\Gamma},J)$ & & & $\nu{\to}\infty$: Gaussian \\
\hline
{Complex generalized Gaussian} 
& $\exp\left(-\frac{t^s}{b}\right)$ 
& $\frac{s\Gamma(m)b^{-m/s}}{\pi^{m}\Gamma(m/s)}$ 
& $s{=}1$: Gaussian\\ ${CGG}_s({\mu},{\Gamma},J)$ & & & $s{=}\frac{1}{2}$: Laplace \\
\hline
{Complex $W$-distribution} 
& $t^{s-1}\exp\left(-\frac{t^{s}}{b}\right)$ 
& $\frac{s\Gamma(m)b^{-\frac{s+m-1}{s}}}{\pi^{m}\Gamma\left(\frac{s+m-1}{s}\right)}$ 
& $s{=}1$: Gaussian\\
${CW}_{s}({\mu},{\Gamma},J)$ & & &\\
\hline
{Complex $K$-distibution} 
& $t^{\frac{\nu-m}{2}}K_{\nu-m}(2\sqrt{\nu t})$ 
& $\cfrac{2\nu^{(\nu+m)/2}}{\Gamma(\nu)\pi^{m}}$ 
& $\nu{\to}\infty$: Gaussian\\
$CK_{\nu}({\mu},{\Gamma},J)$ & & &\\
\hline
\end{tabular}
\end{table}
In Table~\eqref{CES}, we summarize representative subclasses of the CES family obtained from different choices of the density generator \(\psi(t)\). The first column lists the distribution name, the second gives its density function, the third provides the corresponding normalizing constant, and the last column identifies notable special cases. Here, \(\nu>0\) denotes the degrees of freedom, \(s>0\) the exponent, and \(b>0\) the scale parameter.

\begin{theorem}\label{theorem CES}
Let $\Phi^{-1}$ denote the inverse CDF of the standard CES distribution.  
The unified probabilistic constraint
\[
\mathbb{P}\!\left[ \operatorname{Re}(Mz)\le m\right]\ge p
\]
is equivalent to the deterministic constraint
\begin{align}
 \operatorname{Re}(\mu_M z) + \Phi^{-1}(p) \sqrt{\tfrac12\!\left(z^H\Gamma_M z+ \operatorname{Re}(z^T J_M z)\right)}\le m.
\end{align}
If $M$ is proper, then $J_M=0$, and the constraint reduces to the SOCP
\[
\tfrac{\Phi^{-1}(p)}{\sqrt{2}}\;\big\|\Gamma_M^{1/2}z\big\|\le m- \operatorname{Re}(\mu_M z).
\]
\end{theorem}
\begin{proof}
The proof appears in \cite{10.1007/978-3-032-13589-6_21}.
\end{proof}
\subsection{Distributionally Robust 3CP}
In contrast to the previous section, the exact probability distribution of the random data is not assumed to be known; instead, only partial information is available. By {distributional robustness}, we mean that the chance constraint $\mathbb{P}\big[\operatorname{Re}(Mz) \le m\big] \ge p$, is required to hold uniformly over an entire family $\mathcal{I}(\mu_M,\Gamma_M,J_M)$ of probability distributions for the data $M$. Equivalently, we consider enforcing the worst-case (distributionally robust) equivalent constraint
\begin{align}
    \inf_{M \sim \mathcal{I}(\mu_M,\Gamma_M,J_M)} 
    \mathbb{P}\big[\operatorname{Re}(Mz) \le m\big] 
    \ge p,
    \label{eq:DRCCP}
\end{align}
where $M\sim \mathcal{I}(\mu_M,\Gamma_M,J_M)$ indicates the ambiguity set, that is the distribution of $M$ belongs to the family $\mathcal{I}$. 

\subsubsection{Ambiguity Set with Known First Two Order Moments}\label{subsec: 5.1}
The first problem that we consider is one where the family $\mathcal{I}(\mu_B,\Gamma_B, J_B)$ is composed of all distributions having given mean $\mu_{B}$, covariance $\Gamma_{B}$ and pseudo-covariance $J_{B}$, i.e., for $M\in\{B,-D\}$, the ambiguity set 
\[
\mathcal{I}(\mu_M,\Gamma_M,J_M)
=\left\{F \;\middle|\;
\begin{aligned}
&\mathbb{E}_{F}[M] = \mu_M,\\[2pt]
&\mathbb{E}_{F}\!\big[(M-\mu_M)(M-\mu_M)^{H}\big] = \Gamma_M,\\[2pt]
&\mathbb{E}_{F}\!\big[(M-\mu_M)(M-\mu_M)^{T}\big] = J_M
\end{aligned}
\right\}.
\]
Where $F$ is a probability distribution of $M$, and $\mathbb{E}_F$ is the expectation operator associated with $F$. 

\begin{theorem}\label{thm:known}
For any $p\in (0,1)$, the distributionally robust 3CP $\inf_{M\sim \mathcal{I}(\mu_{M},\Gamma_{M}, J_{M})}\mathbb{P}[\operatorname{Re}(Mz)\le m]\ge p,$ is equivalent to the following constraint:
\begin{align}
    \operatorname{Re}(\mu_{M}z)+\sqrt{\frac{p}{1-p}}\sqrt{\frac{1}{2}(z^H\Gamma_{M}z+\operatorname{Re}(z^TJ_{M}z))}\le m,\label{known mean and var}
\end{align}
This constraint is a convex second-order cone constraint.
\end{theorem}
\begin{proof}
We first rewrite the chance constraint using complements:
\[
\mathbb{P}\!\left( \operatorname{Re}(Mz)\le m\right)\ge p
\;\Longleftrightarrow\;
\mathbb{P}\!\left( \operatorname{Re}(Mz)\ge m\right)\le 1-p
\;\Longleftrightarrow\;
\mathbb{P}\!\left( \operatorname{Re}(Mz)-\mu_{ \operatorname{Re}(Mz)}
\;\ge\; m-\mu_{ \operatorname{Re}(Mz)}\right)
\le 1-p.
\]
If $m<\mu_{ \operatorname{Re}(Mz)}$, then 
$\mathbb{P}[ \operatorname{Re}(Mz)\le m]=0$, which makes the constraint infeasible. Define $t := m-\mu_{ \operatorname{Re}(Mz)} \;\ge\; 0$. By Cantelli's (one-sided Chebyshev) inequality,
\[
\mathbb{P}\!\left( \operatorname{Re}(Mz)-\mu_{ \operatorname{Re}(Mz)}\ge t\right)\;\le\;\frac{\Gamma_{ \operatorname{Re}(Mz)}}{\Gamma_{ \operatorname{Re}(Mz)}+t^{2}}.
\]
To ensure the chance constraint holds for all distributions with the same mean 
and variance, it suffices to require $\frac{\Gamma_{ \operatorname{Re}(Mz)}}{\Gamma_{ \operatorname{Re}(Mz)}+t^{2}}\;\le\; 1-p$. Solving for $t$ yields $t \;\ge\;\sqrt{\Gamma_{ \operatorname{Re}(Mz)}}\sqrt{\frac{p}{1-p}}$. Since $t=m-\mu_{ \operatorname{Re}(Mz)}$, we obtain the equivalent deterministic constraint
\[\mu_{ \operatorname{Re}(Mz)}+\sqrt{\Gamma_{\operatorname{Re}(Mz)}}\sqrt{\frac{p}{1-p}}\;\le\; m,\]
which completes the proof.
\end{proof}

\subsubsection{Ambiguity Set with Unknown Second Order Moment}
In this section, we consider the ambiguity set which accounts for the information about the mean vector $\mu_M$ for the vector $M$, and an upper bound on the covariance matrix of the pair $(M_R,M_I)$, $L_M\succeq 0$, with $M=M_R+iM_I$. Define the ambiguity set as 
\begin{align}
\mathcal{I}(\mu_M,L_M)
=\left\{F \;\middle|\;
\begin{aligned}
&\mathbb{E}_{F}[M] = \mu_M,\\[2pt]
&
\begin{bmatrix}
    \Gamma_{M_R} & \Gamma_{M_R,M_I}\\
    \Gamma_{M_I,M_R} & \Gamma_{M_I}
\end{bmatrix}
\preceq 
\begin{bmatrix}
    \hat{\Gamma}_{M_R} & \hat{\Gamma}_{M_R,M_I}\\
    \hat{\Gamma}_{M_I,M_R} & \hat{\Gamma}_{M_I}
\end{bmatrix}=L_M
\end{aligned}
\right\}.
\end{align}
In this case, we define
\[\hat{\Gamma}_M=\hat{\Gamma}_{M_R}+\hat{\Gamma}_{M_I}+i\left(\hat{\Gamma}_{M_I,M_R}-\hat{\Gamma}_{M_R,M_I}\right)\]
\[\hat{J}_M=\hat{\Gamma}_{M_R}-\hat{\Gamma}_{M_I}+i\left(\hat{\Gamma}_{M_I,M_R}+\hat{\Gamma}_{M_R,M_I}\right)\]
\begin{theorem}
For any $p\in (0,1)$, the distributionally robust 3CP $\inf_{M\sim \mathcal{I}(\mu_{M},\hat{\Gamma}_{M}, \hat{J}_{M})}\mathbb{P}[\operatorname{Re}(Mz)\le m]\ge p$, is equivalent to the following constraint:
\begin{align}
    \operatorname{Re}(\mu_{M}z)+\sqrt{\frac{p}{1-p}}\sqrt{\frac{1}{2}(z^H\hat{\Gamma}_{M}z+\operatorname{Re}(z^T\hat{J}_{M}z))}\le m,\label{unknown var}
\end{align}
\end{theorem}
\begin{proof}
    Since $\sqrt{\frac{1}{2}(z^H\hat{\Gamma}_{M}z+\operatorname{Re}(z^T\hat{J}_{M}z))}=\sqrt{\tilde{z}^TL_M\tilde{z}}\geq \sqrt{\tilde{z}\Sigma_M\tilde{z}}$, where $\tilde{z}=[z_R~z_I]^T,\Sigma_M$ is the true covariance of the pair $(M_R,M_I)$, then the worst-case probability is attained when the covariance equal to maximum allowed covariance $L_M=\Sigma_M$. Then we follow Theorem \eqref{thm:known} and this concludes the proof.
\end{proof}
\subsubsection{Ambiguity Set with Unknown Moments} 
We consider the case where the mean vector of $M$ lies in an ellipsoid of size $\zeta_{M}\ge0$ centered at $\mu_M$, and the covariance of $M=M_R+iM_I$ satisfies $\mathrm{Cov}[(M_R,M_I)]\preceq L_M$ for some PSD matrix $L_M$. The ambiguity set is
\begin{align}
\mathcal{I}(\mu_M,L_M)
=\left\{F \;\middle|\;
\begin{aligned}
&(\mathbb{E}[M]-\mu_M)^H\hat{\Gamma}_M^{-1}(\mathbb{E}[M]-\mu_M)\le\zeta_{M}\\
&
\begin{bmatrix}
    \Gamma_{M_R} & \Gamma_{M_R,M_I}\\
    \Gamma_{M_I,M_R} & \Gamma_{M_I}
\end{bmatrix}
\preceq 
\begin{bmatrix}
    \hat{\Gamma}_{M_R} & \hat{\Gamma}_{M_R,M_I}\\
    \hat{\Gamma}_{M_I,M_R} & \hat{\Gamma}_{M_I}
\end{bmatrix}=L_M
\end{aligned}
\right\}.
\end{align}
where 
\[\hat{\Gamma}_M=\hat{\Gamma}_{M_R}+\hat{\Gamma}_{M_I}+i(\hat{\Gamma}_{M_I,M_R}-\hat{\Gamma}_{M_R,M_I})\]
\[\hat{J}_M=\hat{\Gamma}_{M_R}-\hat{\Gamma}_{M_I}+i(\hat{\Gamma}_{M_I,M_R}+\hat{\Gamma}_{M_R,M_I}).\]

\begin{theorem}
For any $p\in(0,1)$, the distributionally robust 3CP $\inf_{M\sim\mathcal{I}(\mu_M,\hat{\Gamma}_M,\hat{J}_M)}\mathbb{P}\!\left[\mathrm{Re}(Mz)\le m\right]\ge p$ is equivalent to
\[
\mathrm{Re}(\mu_M z)+\sqrt{\zeta_{M}}\sqrt{z^H\hat{\Gamma}_M z}+\sqrt{\tfrac{p}{1-p}}\sqrt{\tfrac12\!\left(z^H\hat{\Gamma}_M z+\mathrm{Re}(z^T\hat{J}_M z)\right)}\le m.
\]
If $M$ is proper, then $J_M=0$, and the constraint reduces to
\[
\left(\sqrt{\tfrac{p}{2(1-p)}}+\sqrt{\zeta_{M}}\right)\sqrt{z^H\hat{\Gamma}_M z}\le m.
\]
\end{theorem}
\begin{proof}
Expanding the worst-case probability,
\[
\inf_{(\mu_M,\Sigma_M)\in\mathcal{U}}\inf_{F\in\mathcal{I}(\mu_M,\Sigma_M)}
\mathbb{P}[\mathrm{Re}(Mz)\le m]\ge p,
\]
where 
$\mathcal{U}=\{(\mu_M,\Sigma_M):(\mu_M-\hat{\mu}_M)^H\hat{\Gamma}_M^{-1}(\mu_M-\hat{\mu}_M)\le\zeta_{M},\;\Sigma_M\preceq L_M\}$.
By one-sided Chebyshev,
\[
\inf_F\mathbb{P}[\mathrm{Re}(Mz)\le m]
=1-\frac{1}{1+\frac{(m-\mathrm{Re}(\mu_Mz))^2}{\frac12(z^H\Gamma_Mz+\mathrm{Re}(z^TJ_Mz))}},
\]
which is nonzero only if $\mathrm{Re}(\mu_Mz)\le m$. Thus
\[
\inf_{(\mu_M,\Sigma_M)\in\mathcal{U}}
\frac{m-\mathrm{Re}(\mu_Mz)}
{\sqrt{\frac12(z^H\Gamma_Mz+\mathrm{Re}(z^TJ_Mz))}}
\ge \sqrt{\tfrac{p}{1-p}}.
\]
Define
\[
h(z)=\min_{\mu_M,\Sigma_M}
\frac{m-\mathrm{Re}(\mu_Mz)}
{\sqrt{\frac12(z^H\Gamma_Mz+\mathrm{Re}(z^TJ_Mz))}},
\quad\text{s.t. }\;
(\mu_M-\hat{\mu}_M)^H\hat{\Gamma}_M^{-1}(\mu_M-\hat{\mu}_M)\le\zeta_{M},\;\Sigma_M\preceq L_M.
\]
The numerator depends only on $\mu_M$ and denominator only on $\Sigma_M$. Thus $h(z)=(m+g_1(z))/\sqrt{g_2(z)}$ with
\[
g_1(z)=\min_{\mu_M}-\mathrm{Re}(\mu_Mz)\;\text{s.t.}\;(\mu_M-\hat{\mu}_M)^H\hat{\Gamma}_M^{-1}(\mu_M-\hat{\mu}_M)\le\zeta_{M},
\]
\[
g_2(z)=\max_{\Sigma_M}\tilde{z}^T\Sigma_M\tilde{z}\;\text{s.t.}\;\Sigma_M\preceq L_M.
\]
KKT gives 
$\mu_M^\star=\hat{\mu}_M+\frac{\sqrt{\zeta_{M}}\hat{\Gamma}_Mz}{\sqrt{z^H\hat{\Gamma}_Mz}}$,  
so 
$g_1(z)=-\mathrm{Re}(\hat{\mu}_Mz)-\sqrt{\zeta_{M}}\sqrt{z^H\hat{\Gamma}_Mz}$.
Since the worst-case covariance is $\Sigma_M=L_M$,
\[
g_2(z)=\tilde{z}^TL_M\tilde{z}
=\tfrac12\!\left(z^H\hat{\Gamma}_Mz+\mathrm{Re}(z^T\hat{J}_Mz)\right).
\]
Substituting $g_1,g_2$ yields
\[
\mathrm{Re}(\mu_Mz)
+\sqrt{\zeta_{M}}\sqrt{z^H\hat{\Gamma}_Mz}
+\sqrt{\tfrac{p}{1-p}}
\sqrt{\tfrac12(z^H\hat{\Gamma}_Mz+\mathrm{Re}(z^T\hat{J}_Mz))}
\le m.
\]
\end{proof}
\subsection{Characterization of the Saddle Point}
\begin{table}
\centering
\small
\caption{Summary of $k_p$ and valid $p$-ranges for different distributional settings.}
\label{tab: all}
\begin{tabular}{|c|c|c|}
\hline
\textbf{Case} & \textbf{$p$-range} & \textbf{$k_p$} \\[3pt]
\hline
{CES distribution} 
& $[0.5,1)$ 
& $\Phi^{-1} \left(p\right)$ \\
\hline
Known $(\mu,\Gamma,J)$
& $(0,1)$ 
& $ \sqrt{\cfrac{p}{1-p}}$ \\
\hline
Unknown second-order moment
& $(0,1)$ 
& $\sqrt{\cfrac{p}{1-p}}$ \\
\hline
Unknown moments (proper case)
& $(0,1)$ 
& $ \sqrt{\cfrac{p}{1-p}}\sqrt{\zeta_2}+\sqrt{\zeta_1}$ \\
\hline
\end{tabular}
\end{table}
A pair $(u^\star,v^\star)\in S_1(p_1)\times S_2(p_2)$ is a {saddle-point equilibrium} if
\begin{align}
\operatorname{Re}(u^HAv^\star)\le \operatorname{Re}(u^{\star H}Av^\star)\le \operatorname{Re}(u^{\star H}Av),\qquad \forall (u,v)\in S_1(p_1)\times S_2(p_2).
\end{align}
with the following equivalent feasible sets
\begin{align}
S_1(p_1)=&\left\{u\in U\;\bigg|\; \operatorname{Re}(\mu_{B}u)+k_{p_1}
\sqrt{\tfrac{1}{2}\big(u^H\Gamma_{B}u+\operatorname{Re}(u^TJ_{B}u)\big)}\le b\right\},\label{eq:P-det}\\[3pt]
S_2(p_2)=&\left\{v\in V\;\bigg|\; \operatorname{Re}(\mu_{D}v)-k_{p_2}
\sqrt{\tfrac{1}{2}\big(v^H\Gamma_{D}v+\operatorname{Re}(v^TJ_{D}v)\big)}\ge d\right\}.\label{eq:D-det}
\end{align}
Where $k_p$ is defined from the table \ref{tab: all}. If $B$ is proper, as well as $D$, then the feasible sets simplify to
\begin{align}
S_1(p_1)=\Big\{u\in U\;\bigg|\;
\tfrac{k_{p_{1}}}{\sqrt{2}}
\big\|\Gamma_{B}^{1/2}u\big\|\le b\Big\},\quad 
S_2(p_2) =\Big\{v\in V\;\bigg|\;
\tfrac{k_{p_{2}}}{\sqrt{2}}
\big\|\Gamma_{D}^{1/2}v\big\|\le -d\Big\}.
\end{align}
To establish the existence of equilibrium in the complex game under 3CP, we first impose a standard regularity assumption ensuring the strict feasibility of the players’ strategy sets.

\begin{theorem}[Existence of the saddle point]\label{thm: existence}
Let Assumption~1 hold. Then the zero-sum game $G(p)$ admits a saddle-point equilibrium for any 
$p=(p_1,p_2)$ whose components lie in the ranges summarized in Table~\ref{tab: all}.
\end{theorem}
\begin{proof}
By Lemma~\eqref{lemm:convexity-preservation} and the norm caps in $U,V$, the sets $S_1(p_1)$ and $S_2(p_2)$ are convex and compact. The payoff $\operatorname{Re}(u^HAv)$ is continuous and bilinear. The minimax theorem applies, yielding the existence of a saddle point.
\end{proof}

Writing $u=u_R+iu_I, \tilde{u}=[u_R~u_I]^T$ and $v=v_R+iv_I, \tilde{v}=[v_R~v_I]^T$, and $B=B_R+iB_I$, $D=D_R+iD_I$, we obtain in the split real–imaginary form used for conic modeling:
\begin{align*}
\tfrac{1}{2} \big(u^H\Gamma_{B}u+\operatorname{Re}(u^TJ_{B}u)\big)
&=u_R^T\Gamma_{B_R}u_R-u_R^T\Gamma_{B_R,B_I}u_I-u_I^T\Gamma_{B_I,B_R}u_R+u_I^T\Gamma_{B_I}u_I,\\
\tfrac{1}{2} \big(v^H\Gamma_{D}v+\operatorname{Re}(v^TJ_{D}v)\big)
&=v_R^T\Gamma_{D_R}v_R-v_R^T\Gamma_{D_R,D_I}v_I-v_I^T\Gamma_{D_I,D_R}v_R+v_I^T\Gamma_{D_I}v_I.
\end{align*}
Let
\[
 K_{1}=
\begin{bmatrix}
 \Gamma_{B_R} & -\Gamma_{B_R,B_I} \\ -\Gamma_{B_I,B_R} & \Gamma_{B_I}   
\end{bmatrix}, \quad 
K_{2}=
\begin{bmatrix}
 \Gamma_{D_R} & -\Gamma_{D_R,D_I} \\ -\Gamma_{D_I,D_R} & \Gamma_{D_I}   
\end{bmatrix},
\]
which are positive semidefinite. We write $K_{1}=Q^TQ$ and $K_{2}=P^TP$, where 
\begin{align}
     Q&= 
    \begin{bmatrix}
        Q_{1} & Q_{2}\\
        Q_{3} & Q_{4}
    \end{bmatrix}, \quad 
    P= 
    \begin{bmatrix}
        P_{1} & P_{2}\\
        P_{3} & P_{4}
    \end{bmatrix},\label{QP}
\end{align}
with $Q_1,Q_2,Q_3,Q_4\in\mathbb{R}^{n\times n}$ and  $P_1,P_2,P_3,P_4\in\mathbb{R}^{m\times m}$. For the primal problem, set
\begin{align}
    \hat{Q}_{1}&=Q_{1}+Q_{4}+i(Q_{3}-Q_{2}),\\
    \hat{Q}_{2}&=Q_{1}-Q_{4}+i(Q_{3}+Q_{2}),\label{Q}
\end{align}
and for the dual problem, set
\begin{align}
    \hat{P}_{1}&=P_{1}+P_{4}+i(P_{3}-P_{2}), \\
    \hat{P}_{2}&=P_{1}-P_{4}+i(P_{3}+P_{2}).\label{P}
\end{align}
We consider the primal–dual game form $G(p)$ with feasible sets $S_1(p_1)$ and $S_2(p_2)$ defined in \eqref{eq:P-det}--\eqref{eq:D-det}. For fixed $v\in S_2(p_2)$, the Lagrangian of the inner maximization over $u$ is
\begin{align}
\mathcal{L}(u;v,\delta_1,\sigma,\rho_1,r_1)
&= \operatorname{Re}(u^HAv)-\operatorname{Re}\big(\rho_1(\mathbf{1}^Tu-1)\big)+r_1^T\operatorname{Re}(u) \nonumber\\
&\quad -\delta_1\Big(\operatorname{Re}(\mu_{B}u)+k_{p_1}
\sqrt{\tfrac{1}{2}\big(u^H\Gamma_{B}u+\operatorname{Re}(u^TJ_{B}u)\big)}-b\Big) 
-\sigma(\|u\|-\alpha_1),
\end{align}
where $\delta,\sigma\in\mathbb{R}_+$ and $\rho_1\in\mathbb{C}$, $r_1\in\mathbb{R}^n$. We know that
\begin{align}
-\delta k_{p_1}\sqrt{\tfrac{1}{2}\big(u^H\Gamma_{B}u+\operatorname{Re}(u^TJ_{B}u)\big)} 
= -\delta k_{p_1}\,\|Q\tilde{u}\|.
\end{align}
Using the dual representation of the norm \eqref{lem: dual norm} and setting $\tilde\lambda=[\lambda_R~\lambda_I]^T$ with $\lambda=\lambda_R+i\lambda_I\in\mathbb{C}^n$, we obtain
\begin{align}
-\delta k_{p_1} \|Q\tilde{u}\|
&=\min_{\|\tilde\lambda\|\le \delta k_{p_1}}-\tilde{\lambda}^TQ\tilde{u}\\
&=\min_{\|\tilde{\lambda}\|\le \delta k_{p_1}}-\big(\lambda_R^TQ_1u_R+\lambda_I^TQ_3u_R+\lambda_R^TQ_2u_I + \lambda_I^TQ_4u_I\big)\\
&=\min_{\|\lambda\|\le \delta k_{p_1}}-\frac{1}{2}\big(\operatorname{Re}(\lambda^H\hat{Q}_{1}u+\lambda^H\hat{Q}_{2}\bar{u})\big)\\
&=\min_{\|\lambda\|\le \delta k_{p_1}}-\frac{1}{2}\big(\operatorname{Re}((\hat{Q}_{1}^H\lambda+\hat{Q}_{2}^T\bar{\lambda})^Hu)\big).
\end{align}
Using Lemma~\eqref{lem: dual norm}, the dual of the norm is
\[
-\sigma\|u\|=\min_{\|\beta\|\le\sigma}-\operatorname{Re}(\beta^Hu).
\]
Substituting and eliminating $u$ gives the following primal SOCP in $v$:
\begin{equation}
\begin{aligned}
\min_{v,\lambda_1,\beta_1,\rho_1,r_1,\delta_1}
&\ \lambda_1^T b+\alpha_1\|\beta_1\|+\operatorname{Re}(\rho_1) \\[4pt]
\text{s.t.}\quad
& Av-\mu_{B}^H\delta_1-\tfrac{1}{2}\big(\hat{Q}_{1}^H\lambda_{1}+\hat{Q}_{2}^T\bar{\lambda}_1\big)-\beta_1+r_1-\bar{\rho}_1\mathbf{1} = 0,\\[2pt]
& \|\lambda_1\|\le\delta_1 k_{p_1}, \quad \delta_1\ge 0,\quad r_1\ge 0,\\
& \|v\|\le \alpha_2,\quad \operatorname{Re}(v)\ge 0,\quad \mathbf{1}^T v = 1,\\[2pt]
& \operatorname{Re}(\mu_{D} v)-k_{p_{2}} \sqrt{\tfrac{1}{2}\big(v^H\Gamma_{D}v+\operatorname{Re}(v^TJ_{D}v)\big)}\ge d.
\end{aligned}
\tag{P3}\label{P3}
\end{equation}
Here $v\in\mathbb{C}^m$, $\lambda_1,\beta_1\in\mathbb{C}^n$, $r_1\in\mathbb{R}^n$, and $\rho_1\in\mathbb{C}$. By symmetry (exchanging the players’ roles), we obtain the dual SOCP in $u$:
\begin{equation}
\begin{aligned}
\max_{u,\lambda_2,\beta_2,\rho_2,r_2,\delta_2}
&\ \lambda_2^T d-\alpha_2\|\beta_2\|-\operatorname{Re}(\rho_2) \\[4pt]
\text{s.t.}\quad
& A^H u-\mu_{D}^H\delta_2+\tfrac{1}{2}\big(\hat{P}_{1}^H\lambda_2+\hat{P}_{2}^T\bar{\lambda}_2\big)+\beta_2 - r_2+\bar{\rho}_2\mathbf{1}= 0,\\[2pt]
& \|\lambda_2\|\le\delta_2 k_{p_2}, \quad \delta_2\ge 0,\quad r_2\ge 0,\\
&\|u\|\le \alpha_1,\quad \operatorname{Re}(u)\ge 0,\quad \mathbf{1}^T u = 1,\\[2pt]
& \operatorname{Re}(\mu_{B} u)+k_{p_1}
   \sqrt{\tfrac{1}{2}\big(u^H\Gamma_{B}u+\operatorname{Re}(u^TJ_{B}u)\big)}
   \le b.
\end{aligned}
\tag{D3}\label{D3}
\end{equation}
Here $u\in\mathbb{C}^n$, $\lambda_2,\beta_2\in\mathbb{C}^m$, $r_2\in\mathbb{R}^m$, and $\rho_2\in\mathbb{C}$.
For the case where the moments are unknown, we set the Lagrangian variables 
$\eta_1 \in \mathbb{C}^n$ and $\eta_2 \in \mathbb{C}^m$. 
We then replace the first constraint of the primal problem~\eqref{P3} 
by the following:
\begin{align}
& Av - \mu_{B}^H \delta_1 - \hat{\Gamma}_{B}^{1/2\,H}\eta_1- \tfrac{1}{2}\big(\hat{Q}_{1}^H\lambda_{1}+\hat{Q}_{2}^T\bar{\lambda}_1\big) - \beta_1 + r_1 - \bar{\rho}_1\mathbf{1} = 0, \\[2pt]
& \|\eta_1\| \le \delta_1 \sqrt{\zeta_{B}}, \quad \delta_1 \ge 0, \quad r_1 \ge 0,
\end{align}
and we replace the first constraint of the dual problem~\eqref{D3} 
by the following:
\begin{align}
& A^H u - \mu_{D}^H \delta_2- \hat{\Gamma}_{D}^{1/2\,H} \eta_2 + \tfrac{1}{2} \big(\hat{P}_{1}^H\lambda_{2} +\hat{P}_{2}^T \bar{\lambda}_{2} \big) + \beta_2 - r_2 + \bar{\rho}_2\mathbf{1} = 0, \\[2pt]
& \|\eta_2\| \le \delta_2 \sqrt{\zeta_{D}}, \quad\delta_2 \ge 0, \quad r_2 \ge 0.
\end{align}

\begin{theorem}\label{thm:KKT-saddle}
Let Assumption~1 hold, and let $p=(p_1,p_2)$ with 
$p_1\in(0.5,1)$ and $p_2\in(0.5,1)$.  
A pair $(u^\star,v^\star)$ is a saddle--point equilibrium of $G(p)$ if and only if there exist Lagrange multipliers $(\lambda_1^\star,\beta_1^\star,\rho_1^\star,r_1^\star)
\text{ and }
(\lambda_2^\star,\beta_2^\star,\rho_2^\star,r_2^\star)$
such that $(v^\star,\lambda_1^\star,\beta_1^\star,\rho_1^\star,r_1^\star)
\text{ and }
(u^\star,\lambda_2^\star,\beta_2^\star,\rho_2^\star,r_2^\star)$
are optimal solutions of the primal--dual pair \eqref{pb: P2}--\eqref{pb: D2}, respectively.
\end{theorem}
\begin{proof}
If $(u^\star,v^\star)$ is a saddle point of $G(p)$, then for $v=v^\star$ the 
maximization is solved by $u^\star$, and for $u=u^\star$ the 
minimization is solved by $v^\star$.  
Assumption~1 ensures strict feasibility for both subproblems, so by conic 
duality, the SOCP reformulations of these problems admit the dual optimal 
multipliers 
$(\lambda_1^\star,\beta_1^\star,\rho_1^\star,r_1^\star)$ and 
$(\lambda_2^\star,\beta_2^\star,\rho_2^\star,r_2^\star)$.

Conversely, assume that 
$(v^\star,\lambda_1^\star,\beta_1^\star,\rho_1^\star,r_1^\star)$ and 
$(u^\star,\lambda_2^\star,\beta_2^\star,\rho_2^\star,r_2^\star)$ 
are optimal in \eqref{pb: P2} and \eqref{pb: D2}.  
Strong duality implies equality of the optimal values:
\begin{equation}
\label{eq:value-match-final}
\lambda_1^{\star T} b + \alpha_1\|\beta_1^\star\| + \operatorname{Re}(\rho_1^\star)
=
\lambda_2^{\star T} d - \alpha_2\|\beta_2^\star\| - \operatorname{Re}(\rho_2^\star).
\end{equation}
Now consider any $u\in S_1(p_1)$.  
By feasibility of $u$ in \eqref{pb: P2}, the Lagrangian inequality yields
\begin{align}
\operatorname{Re}(u^H A v^\star)\le \lambda_1^{\star T} b + \alpha_1\|\beta_1^\star\| + \operatorname{Re}(\rho_1^\star).\label{eq:left-bound}
\end{align}
Similarly, for every $v\in S_2(p_2)$, feasibility of $v$ in \eqref{pb: D2} gives
\begin{align}
\lambda_2^{\star T} d - \alpha_2\|\beta_2^\star\| - \operatorname{Re}(\rho_2^\star)
\le \operatorname{Re}(u^{\star H} A v).
\label{eq:right-bound}
\end{align}
Combining \eqref{eq:left-bound}--\eqref{eq:right-bound} with 
\eqref{eq:value-match-final}, valid for all feasible $(u,v)$, yields
\[
\operatorname{Re}(u^H A v^\star)\le \operatorname{Re}(u^{\star H} A v),\qquad \forall (u,v)\in S_1(p_1)\times S_2(p_2).
\]
Evaluating this at $(u,v)=(u^\star,v^\star)$ produces equality, which is exactly the saddle-point inequality:
\[
\operatorname{Re}(u^H A v^\star)
\le
\operatorname{Re}(u^{\star H} A v^\star)
\le
\operatorname{Re}(u^{\star H} A v).
\]
Thus $(u^\star,v^\star)$ is a saddle point of $G(p)$.
\end{proof}
\section{Numerical Results}\label{sec:5}
In this section, we evaluate the proposed framework through numerical experiments implemented in Python~3 using the \texttt{CVXPY} library \cite{JMLR:v17:15-408, Agrawal02012018}. All experiments were conducted on a machine equipped with an 11th-generation Intel Core i7--1185G7 processor operating at 3.00~GHz and 32~GB of RAM. We begin by introducing a transmitter--jammer waveform interaction application, which serves to illustrate the behavior of the proposed complex-valued formulation and to assess the performance of the solution method in computing equilibrium strategies under complex-valued payoffs. We then model and solve Problems~\eqref{P3} and~\eqref{D3} under a variety of scenarios and parameter settings, and analyze the sensitivity of the resulting equilibrium strategies to key modeling parameters. Our evaluation focuses on three metrics: the behavior of the game under the 3CP framework across different probability levels in terms of the achieved game value and the associated CPU time, the empirical calibration of the chance constraints as measured by the frequency of constraint violations in Monte Carlo simulations, and the impact of varying the parameter $\alpha$ on the players’ equilibrium strategies under different probability levels.

\subsection{Application: Transmitter--Jammer Waveform Interaction}
We consider a waveform-level zero-sum interaction between a transmitter (Player~1) and a jammer (Player~2). Both players transmit complex baseband sequences of length $N$ and are constrained to constant-modulus signaling in order to isolate phase-dependent interference effects. The transmitter selects a waveform $T_i \in \mathcal{T} \subset \mathbb{C}^N$, while the jammer selects a waveform $J_j \in \mathcal{J} \subset \mathbb{C}^N$, subject to the constant-modulus constraint
\[
|T_{i,n}| = |J_{j,n}| = \frac{1}{\sqrt{N}},
\qquad n=1,\dots,N.
\]
This normalization ensures equal energy across strategies and restricts both players to pure phase modulation. Following the matched-filter model in Section~4.4 of \cite{Lu2018CognitiveRW}, after down-conversion and sampling the received baseband signal is written as
\[
\boldsymbol r = \alpha_s \boldsymbol T_i + \alpha_c \boldsymbol J_j + \boldsymbol v,
\]
where $\alpha_s$ accounts for propagation/backscattering from the true target, $\alpha_c$ captures the jammer-to-receiver complex gain, and $\boldsymbol v \sim \mathcal{CN}(\boldsymbol 0,\sigma_v^2 I_N)$ is circular complex Gaussian noise. The discrete-time matched filter is constructed from the reference waveform $T_i$. Its impulse response at lag $k$ is
\[
\boldsymbol h_i(k) = \big[(T_i)_{1-k},\dots,(T_i)_{N-k}\big]^H,
\qquad k=-N+1,\dots,N-1,
\]
where we adopt the standard zero-padding convention $(T_i)_m=0$ for $m\notin\{1,\dots,N\}$.
The matched-filter output at lag $k$ is then
\[
z_k(i,j) = \boldsymbol r^T \boldsymbol h_i(k)
= \sum_{n=1}^{N} r_n\, (T_i)_{n-k}^*.
\]
To isolate the waveform-level interference mechanism, we focus on the jammer-induced contribution by setting
\[
\boldsymbol r^{(\mathrm{jam})} = \alpha_c \boldsymbol J
\]
(and omitting noise for clarity). Substituting into the matched-filter expression yields
\[
z_k^{(\mathrm{jam})}(i,j)
= \alpha_c \sum_{n=1}^{N} (J_j)_n \, (T_i)_{n-k}^*.
\]
We define the complex payoff as the jammer matched-filter output at zero lag,
\[
A_{ij} \triangleq \frac{1}{\alpha_c}\, z^{(\mathrm{jam})}_0(i,j)
= \sum_{n=1}^{N} (J_j)_n\, \overline{(T_i)_n}
= T_i^H J_j,
\]
where the scaling by $\alpha_c$ is removed for normalization. The payoff $A_{ij}$ is the complex interference induced by the jammer at the matched-filter output when the receiver correlates with $T_i$, with $Re(A_{ij})$ and $Im(A_{ij})$ corresponding to the in-phase and quadrature components under that reference.
\begin{figure}
    \centering
    \includegraphics[width=0.65\linewidth]{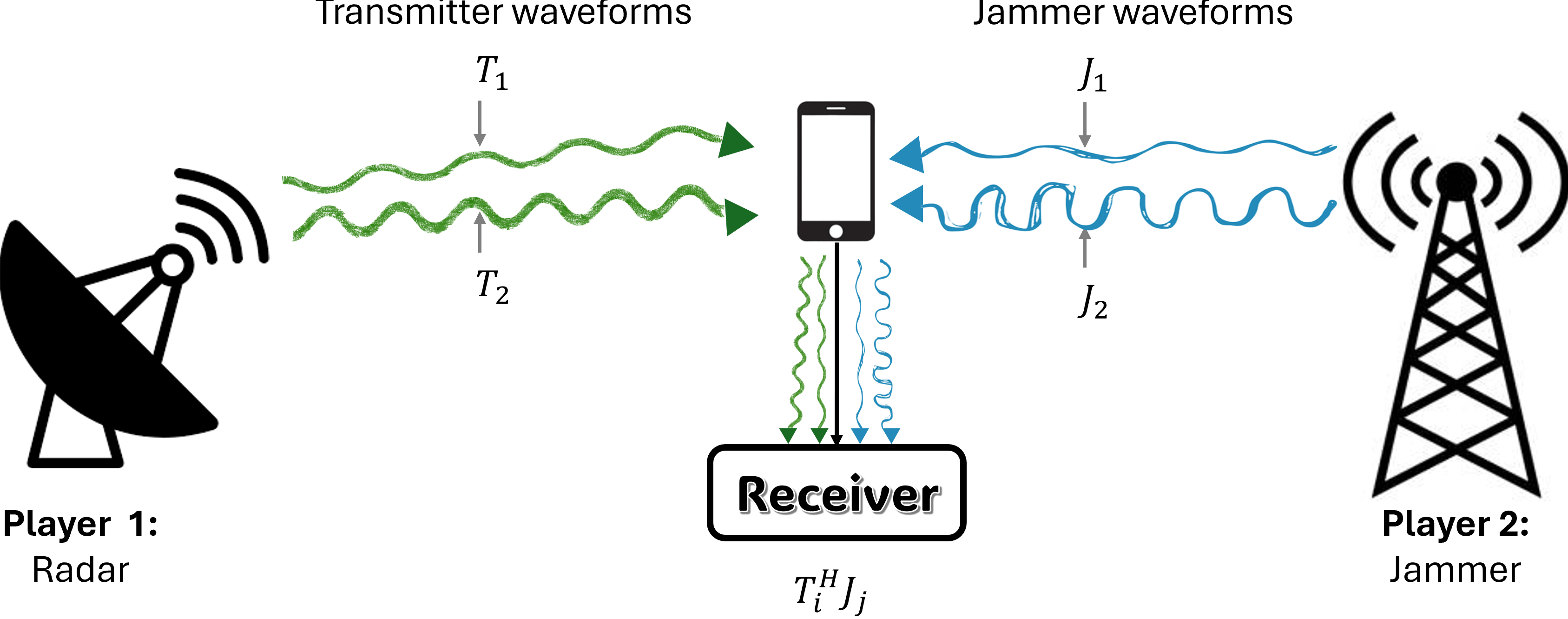}
    \caption{Waveform-level interaction between a transmitter (Player~1) and a jammer (Player~2).}
    \label{fig:tx-jam}
\end{figure}
As illustrated in Fig.~\ref{fig:tx-jam}, the transmitter (Player~1) and the jammer (Player~2) independently select waveforms from their respective strategy sets. The receiver is designed to detect the transmitter signal and therefore constructs a matched filter using the selected transmitter waveform \(T_i\). The jammer waveform \(J_j\) also passes through this matched filter and generates interference at the output. The strength and nature of this interference depend on how closely the jammer waveform aligns with the transmitter waveform. This alignment is quantified by the complex inner product \(T_i^H J_j\). When the two waveforms are well aligned in phase, the jammer produces strong coherent interference at the matched-filter output. When they are poorly aligned, the interference is weaker. Since the waveforms are complex-valued, the interference has both a magnitude and a phase, corresponding respectively to signal amplification and phase distortion effects. We illustrate the model with a concrete instance. Let $N=6$, so that $|T_{i,n}|=|J_{j,n}|=1/\sqrt{6}\approx 0.4082$. Consider two transmitter waveforms $\mathcal{T}=\{T_1,T_2\}$ and two jammer waveforms $\mathcal{J}=\{J_1,J_2\}$ given by
\[
T_1 =
\begin{bmatrix}
0.408 \\
0.204+0.353\,i \\
-0.204+0.353\,i \\
-0.408 \\
-0.204-0.353\,i \\
0.204-0.353\,i
\end{bmatrix},
\quad
T_2 =
\begin{bmatrix}
0.408 \\
0.408\,i \\
-0.408 \\
-0.408\,i \\
0.288+0.288\,i \\
-0.288-0.288\,i
\end{bmatrix},
\]\[
J_1 =
\begin{bmatrix}
0.353+0.204\,i \\
-0.2041+0.353\,i \\
-0.408\,i \\
0.204+0.353\,i \\
-0.204-0.353\,i \\
-0.408
\end{bmatrix},
\quad
J_2 =
\begin{bmatrix}
0.2887-0.2887\,i \\
0.2041+0.3536\,i \\
0.4082\,i \\
-0.4082\,i \\
-0.4082 \\
0.353+0.204\,i
\end{bmatrix}.
\]
Using the definition $A_{ij}=T_i^H J_j$, the resulting complex payoff matrix is
\[
A =
\begin{bmatrix}
0.0833 + 0.0223\,i & 0.5122 - 0.0122\,i \\
0.1012 + 0.2557\,i & 0.1500 - 0.2069\,i
\end{bmatrix}.
\]
Large real parts of $A_{ij}$ correspond to strong coherent interference that directly increases the matched-filter output magnitude, while large imaginary parts correspond to quadrature-phase interference that rotates the detection statistic and degrades coherent integration. The transmitter seeks to select waveforms that minimize both effects, whereas the jammer seeks to maximize them, resulting in a naturally complex-valued zero-sum interaction. To find the best mixed strategies, we solve the primal problem \eqref{pb: P1} and the dual one \eqref{pb: D1}, and we set $\alpha=[0.5,0.5]$, with optimal output: $u^\star=[0.652-0.326i,~ 0.348+0.326i],v^\star=[0.837+0.127i,~0.163-0.127i]$. and optimal $u^{\star H}Av^\star=0.184+0.068i$.

\subsection{Sensitivity Analysis of the Game under 3CP Across Probability Levels.}  
In this section, we investigate how the equilibrium game value produced by the 3CP-based SOCP formulation evolves as the componentwise confidence levels increase. We report results across several problem sizes and distributional assumptions, and we complement these value trends with an empirical calibration study of chance-constraint violations as well as a sensitivity analysis with respect to the norm-bound parameter~$\alpha$.

\subsubsection{Equilibrium game values across sizes, probability levels, and distributional models.}
We consider multiple problem sizes characterized by tuples $(n,l,l_c)$ and $(m,q,q_c)$, selected from $\{(10,5,2),$\\
$(30,15,5),(50,20,7),(100,30,10),(200,50,15)\}$ for Player~1 and Player~2, respectively. Here, $n$ and $m$ denote the number of actions, $l$ and $q$ the total number of linear constraints, and $l_c$ and $q_c$ the number of chance constraints. For each problem size, we set identical componentwise confidence levels $p_1=p_2 \in \{0.60,0.70,0.80,0.95\}$ and evaluate several distributional assumptions, including Gaussian, Laplace, Logistic, Student's $t$-distribution, moment-based, and unknown-moment models. For every instance, we report the optimal game value obtained from Problems~\eqref{pb: P2} and~\eqref{pb: D2}. Table~\eqref{tab:mixed_socp_results} summarizes the resulting game values. Across all problem sizes and probability levels, the equilibrium values exhibit consistent and smooth behavior as $p_1=p_2$ increases. As the problem dimension grows, the sensitivity of the game value to the underlying distribution diminishes, indicating increased robustness of the proposed framework in high-dimensional settings.
\begin{table}[ht]
\caption{Game values under different distributional assumptions and probability levels}
\label{tab:mixed_socp_results}
\centering
\setlength{\tabcolsep}{3pt}
\renewcommand{\arraystretch}{1.0}
\begin{tabular}{|c|c|c| ccccccc|}
\toprule
$(n,l,l_c)$ & $(m,q,q_c)$ & $p_1{=}p_2$ & Cauchy & Gaussian & Laplace & Logistic & $t$-dist & \makecell{Moment\\based} & \makecell{Unknown\\moments} \\
\midrule
(10,5,2)    & (10,5,2)    & 0.60 & -33.76 & -33.57 & -33.49 & -33.98 & -33.59 & -36.33 & -43.02 \\
            &             & 0.70 & -34.89 & -34.33 & -34.29 & -35.24 & -34.38 & -37.20 & -42.67 \\
            &             & 0.80 & -36.78 & -35.22 & -35.43 & -36.81 & -35.33 & -38.42 & -42.13 \\
            &             & 0.95 & -39.55 & -37.52 & -39.09 & -40.37 & -37.97 & -41.36 & -38.64 \\
\midrule
(30,15,5)   & (30,15,5)   & 0.60 &   7.44 &   7.48 &   7.50 &   7.39 &   7.48 &   6.86 &   0.63 \\
            &             & 0.70 &   7.19 &   7.32 &   7.33 &   7.12 &   7.31 &   6.63 &   0.32 \\
            &             & 0.80 &   6.74 &   7.12 &   7.07 &   6.74 &   7.10 &   6.25 &  -0.11 \\
            &             & 0.95 &   0.13 &   6.54 &   5.95 &   4.88 &   6.40 &   2.20 &  -1.44 \\
\midrule
(50,20,7)   & (50,20,7)   & 0.60 &  -0.07 &  -0.09 &  -0.10 &  -0.04 &  -0.09 &   0.17 &  -0.81 \\
            &             & 0.70 &   0.05 &  -0.01 &  -0.01 &   0.08 &  -0.00 &   0.18 &  -0.93 \\
            &             & 0.80 &   0.18 &   0.08 &   0.10 &   0.18 &   0.09 &   0.19 &  -1.10 \\
            &             & 0.95 &  -0.96 &   0.19 &   0.16 &   0.03 &   0.19 &  -0.33 &  -1.90 \\
\midrule
(100,30,10) & (100,30,10) & 0.60 &  -1.80 &  -1.80 &  -1.81 &  -1.79 &  -1.80 &  -1.76 &  -1.81 \\
            &             & 0.70 &  -1.77 &  -1.78 &  -1.79 &  -1.77 &  -1.78 &  -1.76 &  -1.80 \\
            &             & 0.80 &  -1.76 &  -1.77 &  -1.76 &  -1.76 &  -1.77 &  -1.77 &  -1.80 \\
            &             & 0.95 &  -1.79 &  -1.76 &  -1.78 &  -1.79 &  -1.76 &  -1.81 &  -1.70 \\
\midrule
(200,50,15) & (200,50,15) & 0.60 &  -0.60 &  -0.60 &  -0.60 &  -0.59 &  -0.60 &  -0.57 &  -0.49 \\
            &             & 0.70 &  -0.58 &  -0.59 &  -0.59 &  -0.58 &  -0.59 &  -0.56 &  -0.49 \\
            &             & 0.80 &  -0.56 &  -0.58 &  -0.58 &  -0.56 &  -0.58 &  -0.55 &  -0.49 \\
            &             & 0.95 &  -0.49 &  -0.56 &  -0.54 &  -0.53 &  -0.55 &  -0.52 &  -0.47 \\
\bottomrule
\end{tabular}
\end{table}
\subsubsection{Empirical calibration of chance constraints.}  
Now, we consider an instance of size $(n,l_c)=(10,5)$ and $(m,q_c)=(10,5)$, where $n$ and $m$ denote the numbers of actions for Player~1 and Player~2, and $l_c$ and $q_c$ are the numbers of chance constraints (so that all constraints are probabilistic). We first compute the saddle--point strategies $(u^\star,v^\star)$. To empirically assess the realization of the chance constraints, we generate $N=100$ i.i.d.\ scenarios per trial and repeat the simulation over $T=10$ independent trials. A scenario is counted as {violated} for a player if at least one of that player’s $k\in\{l_c,q_c\}$ chance constraints is violated. Figure~\eqref{fig:violated} reports the resulting violation statistics. For the fixed confidence level $p=0.95$, the left panels display the scenario-wise residuals for each constraint (violations occur when the residual is positive; the horizontal line indicates zero), together with the number of violated scenarios in the shown trial. The middle panels summarize the {per-constraint violation ratios} across scenarios, averaged over the $T$ trials, with error bars showing the standard deviation; the dashed red line marks the target level $1-p$. The right panels show the mean and maximum per-constraint violation rates (in \%) as functions of $p$, with the shaded band indicating the standard deviation across trials and the dotted line representing the theoretical target $100(1-p)$. Overall, the empirical violation ratios remain close to the prescribed targets across all tested confidence levels, indicating that the SOCP-based formulation produces well-calibrated chance constraints in practice.
\begin{figure}[ht]
    \centering
    \includegraphics[width=\linewidth]{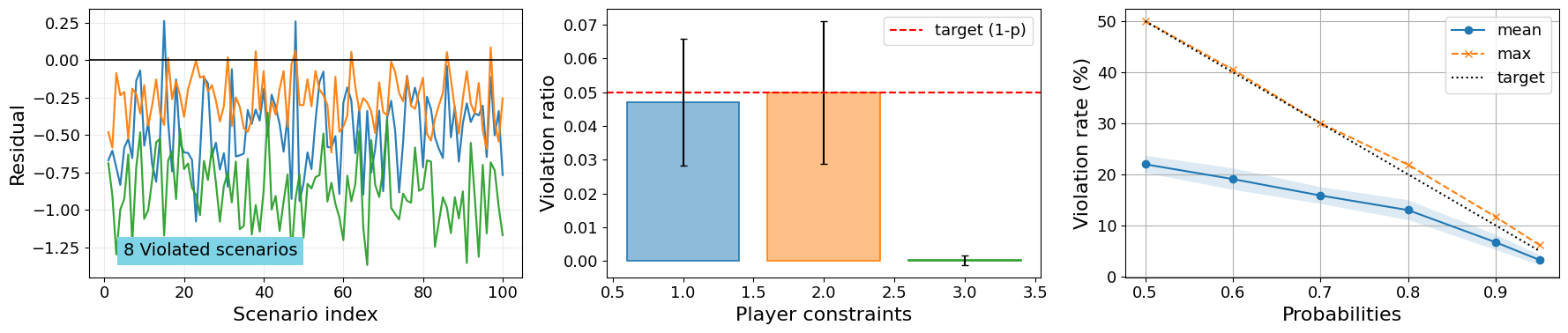}
    \includegraphics[width=\linewidth]{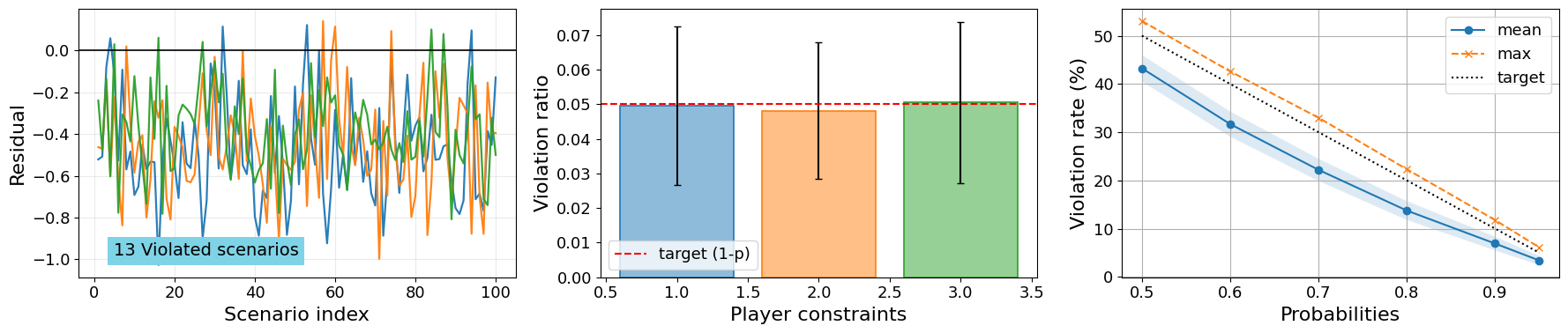}
    \caption{Empirical validation of chance constraints for Player~1 (top row) and Player~2 (bottom row). Left: residuals across scenarios. Middle: per-constraint violation ratios with variability across Monte Carlo trials, compared to the target $1-p$. Right: mean and maximum violation rates versus $p$, shown against the theoretical target.}
    \label{fig:violated}
\end{figure}
\subsubsection{Sensitivity Analysis with Different Values of \texorpdfstring{$\alpha$}{alpha}}
To further investigate how the equilibrium strategies depend on the norm-bound parameter $\alpha$, we examine the Euclidean norms $\|\operatorname{Re}(u^\star)\|$ and $\|\operatorname{Re}(v^\star)\|$ as $\alpha$ varies. Figure~\eqref{fig:alpha_per_p} illustrates the effect of different confidence levels $p$ at a fixed problem size $(n,l_c)=(10,5),(m,q_c)=(10,5)$, while Figure~\eqref{fig:alpha_per_size} compares multiple problem sizes at a fixed confidence level $p=0.9$. In both cases, the equilibrium norms remain bounded and exhibit consistent scaling. From the figures, we observe that the equilibrium norms $|\operatorname{Re}(u^\star)|$ and $|\operatorname{Re}(v^\star)|$ increase with the norm bound parameter $\alpha$ until reaching a saturation point, beyond which further increases in $\alpha$ have little effect. The saturation level depends on the confidence level $p$, with tighter chance constraints (larger $p$) leading to slightly smaller equilibrium norms and lower game values.
\begin{figure}[ht]
    \centering
    \includegraphics[width=1\linewidth]{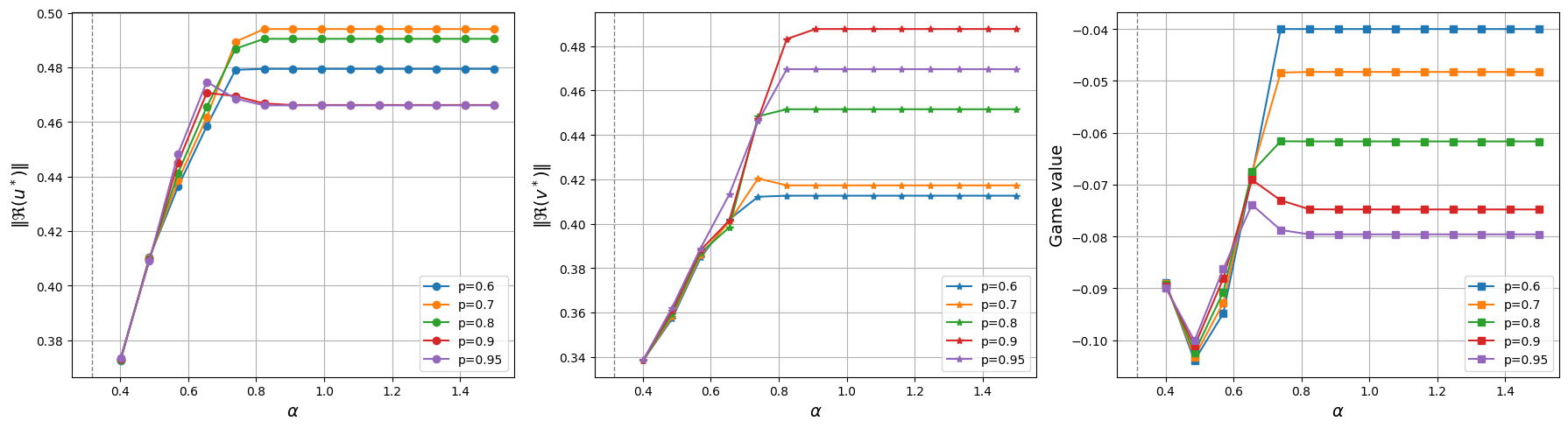}
    \caption{Euclidean norms of the real parts of the equilibrium strategies $\|\operatorname{Re}(u^\star)\|$ (left), $\|\operatorname{Re}(v^\star)\|$ (middle), and game value (right) as functions of the norm bound parameter $\alpha$, for different confidence levels $p$ at fixed size $(n,l_c)=(10,5),(m,q_c)=(10,5)$.}
    \label{fig:alpha_per_p}
\end{figure}
\begin{figure}[ht]
    \centering
    \includegraphics[width=1\linewidth]{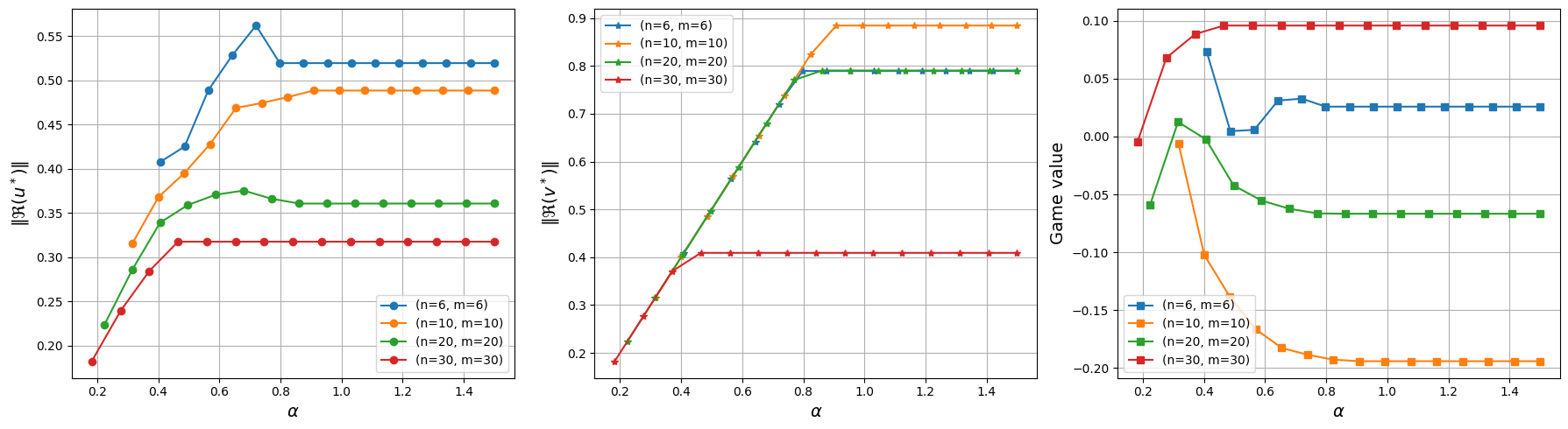}
    \caption{Euclidean norms of the real parts of the equilibrium strategies  $\|\operatorname{Re}(u^\star)\|$ (left) and $\|\operatorname{Re}(v^\star)\|$ (middle), and game value (right) as functions of the norm bound parameter $\alpha$, for different problem sizes at fixed confidence level $p=0.9$.}
    \label{fig:alpha_per_size}
\end{figure}
\section{Conclusion} \label{sec:6}
In this paper, we extend classical zero-sum games to the complex domain by allowing strategies with both probabilistic and phase components. The real part of each strategy encodes probabilities over actions, while the imaginary part, bounded by norm constraints, provides additional flexibility to capture interference or correlation effects. By incorporating chance constraints, we explicitly address uncertainty in the environment and show that the feasible sets remain convex, ensuring that equilibrium computation is tractable. A primal-dual formulation based on duality of the norm yields an effective solution method, supported by both theoretical guarantees and numerical evidence.  The proposed framework is not restricted to two-player interactions. The construction of mixed strategies extends naturally to $n$-player games with general utility functions, allowing applications well beyond the zero-sum setting. In addition, the framework accommodates not only linear but also nonlinear convex constraints without loss of tractability. This adaptability makes the approach suitable for a broad class of optimization-based game formulations, including those that combine probabilistic constraints with structural regularizers such as sparsity or conic restrictions. Our experiments confirm both the theory and the practical reliability of the proposed game under the 3CP framework. First, for different sizes and probability levels, the game values and equilibrium norms stay bounded, showing that saddle–point solutions remain stable even when chance constraints become tighter. Second, the calibration study shows that violation rates match the theoretical targets $(1-p)$ closely, with maximum violations staying within control, which supports the accuracy of the SOCP reformulation in practice. Finally, the analysis with respect to $\alpha$ shows that this parameter only affects the equilibrium for small values, after which both strategies and game values level off, with the final plateau mainly determined by the probability level and the problem size.

\bibliographystyle{plainnat}
\bibliography{bib}

\end{document}